\documentclass[12pt,twoside,reqno]{amsart}

\usepackage[OT1]{fontenc}
\usepackage{type1cm}

\usepackage{comment}
\usepackage{enumerate}
\usepackage{mathabx}
\usepackage{amsthm}
\RequirePackage{amsmath,amsfonts,amssymb,amsthm}

\RequirePackage[dvips]{graphicx}

\usepackage{psfrag}
\usepackage[usenames]{color}
  \numberwithin{equation}{section}

\usepackage[active]{srcltx}  

  \newcommand{\N}{\mathbb{N}}         
  \newcommand{\Z}{\mathbb{Z}}         
  \newcommand{\R}{\mathbb{R}}         
  \newcommand{\C}{\mathbb{C}}         
  \newcommand{\PP}{\mathbb{P}}         

  \newcommand{\supp}{\text{supp}}        

  \newcommand{\la}{\langle}
  \newcommand{\ra}{\rangle}

  \newcommand{\hdim}{\dim_{\mathsf{H}}}
  \newcommand{\fdim}{\dim_{\mathsf{F}}}

\newcommand{\iii}{\mathtt{i}}
\newcommand{\jjj}{\mathtt{j}}

  \newcommand{\e}{\varepsilon}

  \newcommand{\lam}{\lambda}

  \newcommand{\eps}{\varepsilon}
  
  \newcommand{\be}{\begin{equation}}
  \newcommand{\ee}{\end{equation}}
  \newcommand{\ba}{\begin{align}}
  \newcommand{\ea}{\end{align}}
  \newcommand{\bas}{\begin{align*}}
  \newcommand{\eas}{\end{align*}}

  \newcommand{\cI}{\mathcal{I}}    
  \newcommand{\cS}{\mathcal{S}}    
  
  \newcommand{\cP}{\mathcal{P}}    
  \newcommand{\cD}{\mathcal{D}}    
  \newcommand{\cH}{\mathcal{H}}
  \newcommand{\cL}{\mathcal{L}}    

  \newcommand{\wh}{\widehat}
  \newcommand{\wc}{\widecheck}

  \newtheorem{thm}{Theorem}[section]
  \newtheorem{mainthm}{Theorem}
  
  \newtheorem{lemma}[thm]{Lemma}
  \newtheorem{prop}[thm]{Proposition}
  \newtheorem{cor}[thm]{Corollary}

  \theoremstyle{remark}
  \newtheorem{rem}[thm]{Remark}
  \newtheorem{rems}[thm]{Remarks}
  
  \newtheorem{defn}[thm]{Definition}

\begin{document}

\title[Absolute continuity of self-similar measures]{Absolute continuity of self-similar measures, their projections and convolutions}

\author{Pablo Shmerkin}
\address[P. Shmerkin]{Department of Mathematics and Statistics\\
Torcuato di Tella University\\
Av. Figueroa Alcorta 7350 (1425), Buenos Aires\\
Argentina}
\email{pshmerkin@utdt.edu}
\thanks{P.S. was supported in part by Project PICT 2011-0436 (ANPCyT)}

\author{Boris Solomyak}
\address[B. Solomyak]{Department of Mathematics\\
University of Washington\\
Box 354350, Seattle\\
WA 98195-4350\\
USA}
\email{solomyak@math.washington.edu}

\thanks{B.S. was supported in part by NSF grant DMS-0968879, and by the Forschheimer Fellowship
and ERC AdG 267259 grant at the Hebrew University of Jerusalem}

\begin{abstract}
We show that in many parametrized families of self-similar measures, their projections, and their convolutions, the set of parameters for which the measure fails to be absolutely continuous is very small - of co-dimension at least one in parameter space. This complements an active line of research concerning similar questions for dimension. Moreover, we establish some regularity of the density outside this small exceptional set, which applies in particular to Bernoulli convolutions; along the way, we prove some new results about the dimensions of self-similar measures and the absolute continuity of the convolution of two measures. As a concrete application, we obtain a very strong version of Marstrand's projection theorem for planar self-similar sets.
\end{abstract}

\maketitle

\section{Introduction and main results}

\subsection{Introduction}

One of the most natural questions one can ask about a measure on Euclidean space is: is it absolutely continuous with respect to Lebesgue measure? If the answer is affirmative, one would like to gain some information about its density. Basic as it is, the question of absolute continuity is very hard to answer for measures of dynamical or arithmetic origin whose construction involves complicated overlaps. One of the most fascinating examples is the family of \emph{Bernoulli convolutions} $\nu_\lam$, defined as the distribution of the random sum $\pm \lam^n$, where the signs are chosen independently with equal probabilities. For $\lam\in (0,1/2)$ it is well-known and easy to see that $\nu_\lam$ is singular, as it is supported on a Cantor set of Hausdorff dimension less than $1$. For $\lam\in (1/2,1)$, Erd\H{o}s \cite{Erdos39} already in 1939 exhibited a countable set $P$ (the reciprocals of Pisot numbers in $(1,2)$) such that $\nu_\lam$ is singular for $\lam\in P$. Up to today it is not known if these are the only parameters for which $\nu_\lam$ fails to be absolutely continuous. Solomyak \cite{Solomyak95} showed that the exceptional set is small: for Lebesgue almost all $\lam\in (1/2,1)$, the measure $\nu_\lam$ is absolutely continuous with an $L^2$ density. Very recently, Shmerkin \cite{Shmerkin13} showed that indeed $\nu_\lam$ is absolutely continuous for all $\lam$ outside of an exceptional set of Hausdorff dimension zero (no new information about densities was obtained).

The goal of this article is to show that the method of \cite{Shmerkin13}, with suitable extensions and modifications, can be adapted to show that, for a number of natural parametrized families of measures on the real line, the set of ``exceptional'' parameters for which the measure is singular is very small. By ``very small'', we mean a bound on the Hausdorff dimension which is smaller by at least $1$ than the dimension of the parameter space. Some of the families to which our results apply include:
\begin{enumerate}
 \item Fairly general parametrized families of homogeneous self-similar measures on the line (which include Bernoulli convolutions as a special case).
 \item Projections of homogeneous self-similar measures on the plane (possibly containing a scaled irrational rotation).
 \item Convolutions of scaled homogeneous self-similar measures on the line (the parameter comes in the scaling).
\end{enumerate}

These results have immediate applications to the Lebesgue measure of self-similar sets, their projections and arithmetic sums. Moreover, we are able to get some information about the densities - this is new even in the case of Bernoulli convolutions. Along the way, we establish a new criterion for the absolute continuity of the convolution of two measures, and prove a partial continuity result for the $L^q$ dimensions of some self-similar measures. These results may be of independent interest.

In a forthcoming article, we obtain analogous results for families of parametrized families in the plane, including complex Bernoulli convolutions (see \cite{SolomyakXu03} and references therein), and fat Sierpi\'{n}ski carpets.

Before stating our results more precisely, we review some definitions and set up some notation. In this article, all measures are understood to be Borel probability measures on some Euclidean space $\R^d$; the space of all of them is denoted by $\cP_d$. The \emph{(lower) Hausdorff dimension} of a measure $\mu\in\cP_d$ is
\[
\hdim \mu = \inf\{\hdim A:\mu(A)>0\},
\]
where $\hdim A$ is the Hausdorff dimension of the set $A$. It is then clear that if $\hdim\mu<d$, then $\mu$ is necessarily singular. When studying parametrized families $\{ \mu_u\}_{u\in U}$ it may happen that $\hdim\mu_u<d$ for obvious reasons for all or part of the parameter space. Although at present this is a vague concept, we refer to this part of the parameter space as the \emph{sub-critical} regime, and we are not interested in it as the measures are then automatically singular. Therefore all of our results will contain an assumption that will ensure that we are in the \emph{super-critical} regime, in which there is at least a chance of obtaining absolute continuity.

Our parametrized families will consist of self-similar measures, or measures constructed from self-similar measures through the geometric operations of projection and convolution. As corollaries, we will obtain also results for self-similar sets and their projections and arithmetic sums. Hence, we review their definition and some of their properties, and set up notation along the way.

Recall that an \emph{iterated function system (IFS)} $\cI$ is a finite collection $\{ f_1,\ldots, f_m\}$ of strictly contractive maps on some Euclidean space $\R^d$. In this article, the maps $f_i$ will always be similarities and this will be assumed from now on. The IFS is termed \emph{homogeneous} if the linear parts of the maps $f_i$ are all equal. It is well-known that there exists a unique nonempty compact set $A=A(\cI)\subset\R^d$, the \emph{attractor} or \emph{invariant set} for $\cI$, such that $A=\bigcup_{i=1}^m  f_i(A)$. The set $A$ is then a \emph{self-similar set}. We say that the \emph{open set condition} holds for $\cI$ if there is a nonempty open set $O$ such that the sets $f_i(O)$ are pairwise disjoint and contained in $O$, and that the \emph{strong separation condition} holds if the sets $f_i(A)$ are pairwise disjoint; the latter condition is stronger.

We denote the open simplex in $\R^m$ by $\PP_m$, i.e.\ $\mathbb{P}_m=\{ (p_1,\ldots,p_m): p_i> 0,\ \sum_{i=1}^m p_i=1\}$. We think of elements of the simplex as probability vectors. Given an IFS $(f_1,\ldots,f_m)$ and a vector $p\in\PP_m$, there is a unique measure $\mu\in\cP_d$ such that
\[
\mu = \sum_{i=1}^m p_i\, f_i\mu,
\]
where we use the notation $f\mu$ to denote the push-forward measure $f\mu(A)=\mu(f^{-1}A)$.

In this article we consider mostly \emph{homogeneous} iterated function systems on $\R^d$ of the form
\[
 \cI = \cI_{T,a} = \{ T x + a_1,\ldots, T x + a_m\},
\]
where $T\in\cS_d$, the family of invertible contractive similarities on $\R^d$, and $a=(a_1,\ldots,a_m)\in\R^{md}$ is a tuple of translation  vectors. We identify $\cS_1$ with $(-1,0)\cup (0,1)$, and $\cS_2$ with $\{ z\in \C: 0 < |z| < 1\}$. The invariant set for $\cI_{T,a}$ will be denoted by $A(T,a)$, and the invariant measure for $\cI_{T,a}$ with weights $p\in \PP_m$ by $\mu(T,a,p)$.

Given a homogeneous IFS $\cI =  (T x+a_i)_{i=1}^m$, its \emph{similarity dimension} is $s(T,m)=\log m/\log(1/r)$, where $r\in (0,1)$ is the contraction ratio of $T$. Note that $s(T,m)$ is  independent of $a$. If a weight $p$ is also given, then the similarity dimension is $s(T,m,p)=h(p)/\log(1/r)$, where $h(p)=-\sum_{i=1}^m p_i\log p_i$ is the \emph{entropy} of $p$. It is well known that $0\le h(p)\le m$ with $h(p)=m$ if and only if $p=(\frac1m,\ldots,\tfrac1m)$. Therefore, $s(T,m)=\sup_{p\in\PP_m} s(T,m,p)$ and the supremum is attained exactly when $p=(\frac1m,\ldots,\tfrac1m)$.

The following result is well known, see e.g. \cite[Corollary 5.2.3 and Theorem 5.2.5]{Edgar98} for a proof.

\begin{prop} \label{prop:sim-dim}
\begin{enumerate}[(i)]
\item For any $T\in\cS_d, a\in\R^{md}$ and $p\in \PP_m$, we have the inequalities $\hdim(A(T,a))\le s(T,m)$ and $\hdim(\mu(T,a,p))\le s(T,m,p)$.
\item If the open set condition holds, then there are equalities $\hdim(A(T,a))= s(T,m)$ and $\hdim(\mu(T,a,p))= s(T,m,p)$.
\end{enumerate}
\end{prop}

\subsection{Parametrized families of self-similar measures}

Our first result concerns real-analytic parametrized families of self-similar measures on the real line. See Section \ref{subsec:projection-IFS} below for the definition of the coding map.

\begin{mainthm} \label{mainthm:parametrized}
Let $\{ (\lam_u,a_u) \}_{u\in U}$ be a real-analytic map $\R^\ell\supset U\to \cS_1\times \R^m$ such that the following non-degeneracy condition holds: for any distinct infinite sequences $\iii,\jjj\in \{1,\ldots,m\}^\N$, there is a parameter $u$ such that $\Pi_u(\iii)\neq \Pi_u(\jjj)$, where $\Pi_u$ is the projection map for the IFS $(\lam_u x+ a_{u,i})_{i\in [m]}$.

Then there exists a set $E\subset U$ of Hausdorff dimension at most $\ell-1$, such that for any $u\in U\setminus E$ and any $p\in\PP_m$ such that $s(\lam_u,m,p)>1$, the measure $\mu(\lam_u,a_u,p)$ is absolutely continuous with respect to Lebesgue measure on the line, and has a density in $L^q$ for some $q=q(u,p)>1$.

In particular, $A(\lam_u,a_u)\subset\R$ has positive Lebesgue measure for all $u\in U\setminus E$ such that $\lambda_u m>1$.
\end{mainthm}

\begin{rems}
\begin{enumerate}
\item Hochman \cite{Hochman13} proved a similar result, with ``$\mu(\lam_u,a_u,p)$ has full dimension'' in place of ``$\mu(\lam_u,a_u,p)$ is absolutely continuous'', but valid for arbitrary self-similar measures (not necessarily homogeneous). Our proof of Theorem \ref{mainthm:parametrized} relies heavily on his result.
\item This result generalizes \cite[Theorem 1.2]{Shmerkin13}, which is the special case in which $\ell=1$, the parameter is the contraction ratio $\lam$ and the translations are fixed. One novelty is the fact that the \emph{translations} are now allowed to depend on the parameter as well. Indeed, the situation in which the contraction $\lam$ is fixed and the translations vary arises naturally when studying projections of self-similar measures on $\R^2$; this will be exploited later.
\item The second novelty, even for the family of Bernoulli convolutions, is that we are able to prove that, outside of a zero-dimensional set of exceptions, the densities are in some $L^q$ space with $q>1$. The value of $q$ depends on the parameter and is not explicit, but nevertheless it is some quantitative information about the density.
\item By applying Theorem \ref{mainthm:parametrized} to the identity map of $U=(0,1)\times \R^m$, we obtain in particular that there exists a set $E\subset U$ of Hausdorff dimension $\le m$, such that if $(\lam,a)\in U\setminus E$ and $p\in\PP_m$ is such that $h(p)>|\log\lam|$, then $\mu(\lam,a,p)$ is absolutely continuous.
\item In general, the bound $\ell-1$ on the dimension of the exceptional set is sharp; in short, this is because ``exact overlaps'' are often a co-dimension $1$ set in parameter space, and they may lead to the similarity dimension dropping below $1$.
\end{enumerate}
\end{rems}

\subsection{Orthogonal projections of self-similar measures}

We now deal with the orthogonal projections of a fixed self-similar measure (or set) in $\R^2$. Let $P_\beta:\R^2\to\R^2$ be the orthogonal projection onto a line making an angle $\beta$ with the $x$-axis. We may then enquire about the absolute continuity of $P_\beta\mu$ as $\beta$ varies. If $\hdim\mu<1$, then also $\hdim(P_\beta\mu)\le \hdim\mu<1$, and therefore $P_\beta\mu$ is singular for all $\beta$: this is the subcritical regime. The supercritical regime corresponds to $\hdim\mu>1$ (the \emph{critical} case $\hdim\mu=1$ is of different nature and will not be considered here). In this case, we have the following classical general result:

\begin{prop} \label{prop:marstrand}
\begin{enumerate}[(i)]
\item Let $A\subset\R$ be a Borel set. If $\hdim A\le 1$, then $\hdim P_\beta A=\hdim A$ for almost all $\beta$, while if $\hdim A>1$, then $P_\beta A$ has positive Lebesgue measure for Lebesgue almost all $\beta\in [0,\pi)$.
\item Let $\mu\in\cP_2$. If $\hdim\mu\le 1$, then $\hdim P_\beta \mu= \hdim \mu$ for almost all $\beta$, while if $\hdim\mu>1$, then $P_\beta\mu$ is absolutely continuous for Lebesgue almost all $\beta\in [0,\pi)$.
\end{enumerate}
\end{prop}

The first part is Marstrand's classical theorem on projections, and the second part is a standard variant; see Hu and Taylor
\cite[Theorem 6.1]{HuTaylor94} for the proof. In general, the exceptional set can be large from the point of view of dimension (full Hausdorff dimension) and topology (dense $G_\delta$). Nevertheless, recently there has been much interest in improving general projection results for specific class of sets and measures; in particular, for \emph{self-similar} sets and measures: see e.g. \cite{PeresShmerkin09, NPS12, HochmanShmerkin12, Hochman13, FalconerJin14, Farkas14}.  However, all of these papers deal with the \emph{dimension} part, and the techniques fall short of directly giving any results on absolute continuity (or positive Lebesgue measure in the case of sets). We have the following result for projections of homogeneous self-similar measures:

\begin{mainthm} \label{mainthm:projection-measures}
Fix $\lam=r\exp(2\pi i\alpha)\in\cS_2, a\in\R^{2m}$ such that the planar IFS $(\lambda x+a_i)_{i=1}^m$ satisfies the strong separation condition. There exists a set $E\subset [0,\pi)$ of zero Hausdorff dimension such that the following holds.
\begin{enumerate}[(i)]
\item  The measure $P_\beta\mu(\lam,a,p)$ is absolutely continuous for any $p\in\PP_m$ such that $r h(p)>1$ and any $\beta\in [0,\pi)\setminus E$. Moreover, the density is in some $L^q$ space with $q=q(\beta,p)>1$.
\item If the rotational part of $\lam$ generates a dense subgroup of rotations (i.e. $\alpha/\pi\notin\mathbb{Q}$) and
\begin{equation} \label{eq:Lq-larger-than-1}
r^{q-1}>\sum_{i=1}^m p_i^q,
\end{equation}
for some $q\in (1,2]$, then for all $\beta\in [0,\pi)\setminus E$ the measure $P_\beta\mu(\lam,a,p)$  has a density in $L^q$.
\end{enumerate}
\end{mainthm}

\begin{rem}
The assumption \eqref{eq:Lq-larger-than-1} is sharp up to the endpoint. Indeed, it is well known that (under the open set condition) the $L^q$ dimension of $\mu(\lam,a,p)$ is given
\begin{equation} \label{eq:Lq-dim-ssm}
D_q(\mu_{\lam,a,p}) = \frac{\log \sum_{i=1}^m p_i^q}{(q-1)\log r}.
\end{equation}
(See Section \ref{sec:abs-cont-of-convolution} for the definition of the $L^q$ dimension. See e.g.\ \cite[Theorem 16]{Riedi93} for the proof of (\ref{eq:Lq-dim-ssm}). We will actually only use this
formula under the strong separation condition.) Since $L^q$ dimension does not increase under projections, and a measure on the line of $L^q$ dimension $<1$ cannot have an $L^q$ density, it follows that a necessary condition for $P_\beta\mu(\lam,a,p)$ to have an $L^q$ density is that $r^{q-1}\ge\sum_{i=1}^m p_i^q$.
\end{rem}

As a consequence of Theorem \ref{mainthm:projection-measures}, we obtain an analogous result valid for \emph{all} self-similar sets on $\R^2$.
\begin{mainthm} \label{mainthm:projection-sets}
Let $A\subset\R^2$ be any self-similar set. Then
\begin{align*}
\hdim( \beta\in [0,\pi): \hdim P_\beta A< \hdim A)) = 0 \quad\text{if } \hdim A\le 1,\\
\hdim( \beta\in [0,\pi): \mathcal{L}_1(P_\beta A)=0 ) = 0 \quad\text{if } \hdim A > 1.
\end{align*}
\end{mainthm}
This should be compared with Marstrand's Theorem (Proposition \ref{prop:marstrand}): in the special case of self-similar sets, the exceptional set is not only of Lebesgue measure zero, but in fact of zero Hausdorff dimension. The part of the statement concerning the case $\hdim A\le 1$ was essentially already known, our contribution is the positive measure part.

\subsection{Convolutions of self-similar measures}

Recall that the convolution of the measures $\mu,\nu\in\cP_d$ is the push down of $\mu\times \nu$ under the map $S(x,y)=x+y$. The dimensions of convolutions of self-similar measures were investigated in \cite{NPS12, HochmanShmerkin12} and, indirectly as a consequence of more general results, in \cite{Hochman13}. Since these results serve both as motivation and as key tools in our investigations, we recall some of them. Let $T_u(x)= ux$ be the map that scales a number by $u$.

\begin{thm}[{\cite{HochmanShmerkin12, Hochman13}}]  \label{thm:dimension-convolution}
For $i=1,2$, let $\lam_i\in  (0,1)$, $a_i\in\R^{m_i}, p_i\in\PP_{m_i}, m_i\ge 2$, and suppose the strong separation condition is satisfied for $(\lambda x+a_{ij})_{j=1}^{m_i}$. Write $\mu_i=\mu(\lam_i,a_i,p_i)$ for simplicity.
\begin{enumerate}[(i)]
\item If $\log \lam_2/\log\lam_1\notin\mathbb{Q}$, then
\begin{equation} \label{eq:convolution-dim}
\hdim(\mu_1*T_u \mu_2)=\min(\hdim\mu_1+\hdim\mu_2,1),
\end{equation}
for all $u\in\R\setminus\{0\}$.
\item In general, without any algebraic assumptions, \eqref{eq:convolution-dim} holds for all $u$ outside of a set of dimension zero which is independent of $p_1,p_2$.
\end{enumerate}
\end{thm}

In fact, the first part holds for more general, not necessarily homogeneous self-similar measures. In the homogeneous case, the analogous result for \emph{correlation} rather than Hausdorff dimension was established in \cite{NPS12} (under the irrationality assumption in the first part); this will be discussed in more detail in Section \ref{subsec:proof-convolutions}. These results again suggest the question of whether, in the super-critical regime (sum of the dimensions $>1$), one can go beyond full dimension and establish absolute continuity of the convolution. One cannot hope to get no exceptions when $\log\lam_2/\log\lam_1$ is irrational, since in \cite[Theorem 4.1]{NPS12} it was shown that
\[
\mu(1/3,(0,1),(\tfrac12,\tfrac12)) * T_u \mu(1/4,(0,1),(\tfrac12,\tfrac12))
\]
is singular for a dense $G_\delta$ set of parameters $u$. On the other hand, in \cite[Theorem 1.3]{Shmerkin13}, it was proved that there exists a zero-dimensional set $E\subset (0,1/2)$ such that for $\lam\in (0,1/2)\setminus E$ and all $\lam'$ such that $\lam'/\lam\notin\mathbb{Q}$, the measure
\[
\mu(\lam,(0,1),(\tfrac12,\tfrac12)) * T_u \mu(\lam',(0,1),(\tfrac12,\tfrac12))
\]
is absolutely continuous (even with an $L^2$ density) for \emph{all} $u\in\R\setminus\{0\}$. The following result in some sense interpolates between these two.

\begin{mainthm} \label{mainthm:convolutions}
For $i=1,2$, let $\lam_i\in (-1,0)\cup (0,1)$, $a_i\in\R^{m_i}, m_i\ge 2$, and assume that
\begin{equation} \label{eq:sum-of-dims-larger-1}
\frac{\log m_1}{|\log\lam_1|}+\frac{\log m_2}{|\log\lam_2|}>1.
\end{equation}
Suppose also that the $(\lam_i x+a_{ij})_{j=1}^{m_i}$, $i=1,2$, satisfy the strong separation condition.
\begin{enumerate}[(i)]
\item There exists a set $E\subset \R$ of Hausdorff dimension zero such that
\[
\mu(\lam_1,a_1,p_1)*T_u\mu(\lam_2,a_2,p_2) \ll \mathcal{L}_1,
\]
and has a density in $L^q$ for some $q=q(\lam_i,a_i,p_i)>1$, for all $u\in\R\setminus E$ and any $p_i\in \PP_{m_i}$ such that
\begin{equation} \label{eq:sum-of-hausd-dims-g1}
\hdim\mu(\lam_1,a_1,p_1)+\hdim\mu(\lam_2,a_2,p_2) > 1.
\end{equation}
\item If, additionally, $\log |\lam_2|/\log|\lam_1|$ is irrational, then
\begin{equation*}
\mu(\lam_1,a_1,p_1)*T_u\mu(\lam_2,a_2,p_2) \in L^q(\R)
\end{equation*}
for all $u\in\R\setminus E$ whenever $q\in (1,2]$, and
\begin{equation} \label{eq:sum-of-corr-dims-g1}
\frac{\log\sum_{j=1}^{m_1}p_{1j}^q}{(q-1)|\log\lam_1|} + \frac{\log\sum_{j=1}^{m_2}p_{2j}^q}{(q-1)|\log\lam_2|} > 1.
\end{equation}
\end{enumerate}
\end{mainthm}

\begin{rems}
\begin{enumerate}
\item The main class of examples to keep in mind, which motivated the theorem, are central Cantor sets and Hausdorff measures on them. Recall that the central Cantor set $C_\lam$ is constructed by starting with the interval $[0,1]$, replacing it by the union $[0,\lam]\cup [1-\lam,1]$, and continuing inductively with the same pattern. These sets are self-similar, and  the self-similar measure $\nu_\lam$ with weights $(1/2,1/2)$ coincides with Hausdorff measure of the appropriate dimension on $C_\lam$. Hence, in particular, the theorem says that when $\hdim C_{\lam_1}+\hdim C_{\lam_2}>1$, the convolution $\nu_{\lam_1} * T_u \nu_{\lam_2}$ is absolutely continuous (and hence also $\mathcal{L}_1(C_{\lam_1}+u C_{\lam_2})>0$) outside of a zero-dimensional set of possible exceptions, with an $L^2$ density when $\lam_1$ and $\lam_2$ are rationally independent.
\item Similarly to \eqref{eq:Lq-larger-than-1} in Theorem \ref{mainthm:projection-measures}, assumption \eqref{eq:sum-of-corr-dims-g1} simply says that the $L^q$ dimension of the product measure $\mu(\lam_1,a_1,p_1)\times\mu(\lam_2,a_2,p_2)$ is strictly larger than $1$ which, up to``strictly'', is also a necessary condition for the convolutions to have an $L^q$ density. Notice also that this condition always holds for $q=2$ if $p_i$ are close enough to the uniform weights in $\PP_{m_i}$, provided \eqref{eq:sum-of-dims-larger-1} holds.
\item Recall that in  \cite[Theorem 4.1]{NPS12} it is shown that $\nu_{1/3} * T_u \nu_{1/4}$ is singular for a dense, $G_\delta$ set of parameters $u$.  This shows that Theorem \ref{mainthm:convolutions} is sharp in the sense that the exceptional set may be uncountable and generic from the topological point of view. To the best of our knowledge, this is the first natural instance (as opposed to an ad-hoc example) in which the set of exceptions to a projection theorem is shown to be uncountable yet of zero Hausdorff dimension. Since dense $G_\delta$ sets have full packing dimension, it also shows that our results are intrinsically about Hausdorff dimension and cannot, in general, be extended to packing dimension (this is in contrast to the dimension results in \cite{Hochman13}, where the exceptional set is shown to be of zero \emph{packing} as well as Hausdorff dimension, and perhaps suggests that the super-critical case is intrinsically harder, at least in some cases).
\end{enumerate}
\end{rems}

Similar to Theorem \ref{mainthm:projection-sets}, as a corollary we obtain a result on sums of arbitrary self similar sets on the line (with no assumptions on homogeneity or separation).
\begin{mainthm} \label{mainthm:convolutions-sets}
If $A_1,A_2$ are self-similar sets on $\R$ such that $\hdim A_1+\hdim A_2>1$, then
\[
\hdim\{u: \mathcal{L}_1(A_1+ u A_2)=0\} = 0.
\]
\end{mainthm}
The above can also be seen in the light of Marstrand's projection Theorem (Proposition \ref{prop:marstrand}), since up to a smooth reparametrization and an affine change of coordinates, $(A_1+ u A_2)_{u\in\R\setminus\{0\}}$ are the orthogonal projections of the product set $A_1\times A_2$ in non-principal directions.

\subsection{Strategy of proofs}

The general strategy of the proofs follows the scheme of \cite{Shmerkin13}, although there are some new ingredients as well. All of the measures we consider can be expressed as an infinite convolution of atomic measures (this is the reason why we need to assume the iterated function systems are homogeneous). This allows us to decompose our measure of interest $\mu_u, u\in U$, as a convolution $\nu_u*\eta_u$, in such a way that:
\begin{enumerate}
\item For $u$ outside of a ``small'' exceptional set $E'$, the measure $\nu_u$ has full dimension for an appropriate notion of dimension (Hausdorff or $L^q$ dimension for a suitable value of $q$). For this we apply, depending on the context, any of a number of recent results on the dimension of measures of dynamical origin \cite{NPS12, HochmanShmerkin12, Hochman13}.
\item In the case where the previous step applies to Hausdorff dimension, we employ a result on the continuity of the dimension of a self-similar measure (which is developed in Section \ref{sec:dimensions-of-ssm} and may be of independent interest), to show that off the same exceptional set $E'$, the measure $\nu_u$ has \emph{almost} full $L^q$-dimension for some $q>1$. This step is a new ingredient, and is what ultimately allows us to gain some information about the densities.
\item For $u$ outside of another zero-dimensional set $E''$, the measure $\eta_u$ has power Fourier decay: $|\eta_u(\xi)| \le C\, |\xi|^{-\sigma}$, for some $C,\sigma>0$ which may depend on $u$. This is achieved via a number of modifications of what has become known as the ``Erd\H{o}s-Kahane argument'', which deals with the special case in which $\eta_u$ is the family of Bernoulli convolutions. The realization that the Erd\H{o}s-Kahane technique is very flexible and can be adapted to such a wide array of situations is perhaps one of the main innovations of this paper.
\item To conclude the proofs, we set $E=E'\cup E''$ and use the following result valid for arbitrary Borel measures: if $\eta$ has power Fourier decay with exponent $\sigma$, and the $L^q$ dimension of $\nu$ is sufficiently close to $1$ (in terms of $\sigma$), then $\nu*\eta$ is absolutely continuous with an $L^q$ density. This is a sharpening of \cite[Lemma 2.1]{Shmerkin13}, and is stated more precisely in Section \ref{sec:abs-cont-of-convolution}.
\end{enumerate}


\section{Notation and preliminaries}
\label{sec:preliminaries}
\subsection{Notation}

We use Landau's $O$ notation: $Y=O(X)$ means $Y\le CX$ for some constant $C>0$. By $Y=\Omega(X)$ we mean $X=O(Y)$ and we write $Y=\Theta(X)$ to denote that both $Y=O(X)$ and $X=O(Y)$ hold. When the implicit constant $C$ depends on some other parameters, this will be denoted by a subscript in the $O$ notation, for example $Y=O_k(X)$ means that $Y< C X$ for a constant $C$ that is allowed to depend on $k$.

The integer interval $\{1,2,\ldots,b\}$ is denoted by $[b]$.

For simplicity of notation, logarithms are always to base $2$.

\subsection{Fourier transforms of self-similar measures} \label{subsec:ssm}

The Fourier transform of $\mu\in\mathcal{P}_d$ is
\[
\wh{\mu}(\xi) = \int e^{i\pi \la x,\xi\ra} d\mu(x),
\]
where $\la \cdot,\cdot \ra$ is the standard inner product on $\R^d$. (We choose this slightly unusual normalization for
technical reasons, which will become apparent in the next section.)
Denote
\[
\cD_d = \left\{ \mu\in\mathcal{P}_d: |\wh{\mu}(\xi)| = O_\mu(|\xi|^{-\sigma}) \text{ for some $\sigma>0$}  \right\}.
\]

Let $\mu=\mu(T,a,p)$ be a self-similar measure for a homogeneous IFS on $\R^d$.  Then $\mu$ can be realized as the distribution of a random sum:
\begin{equation} \label{eq:convolution-random-sum}
\mu \sim \sum_{n=1}^\infty T^{n-1} A_n,
\end{equation}
where the $A_n$ are i.i.d.\ Bernoulli random variables with $P(A_n=a_j)=p_j$. This shows that $\mu$ is an infinite convolution of Bernoulli random variables, and we obtain the following infinite product formula for its Fourier transform (which can also be deduced by inductively applying the definition of self-similarity):

\begin{equation} \label{eq:FT-ssm}
\widehat{\mu}(\xi) = \prod_{n=0}^\infty \Phi_n(\xi),
\end{equation}
where
\begin{equation*}
\Phi_n(\xi) = \sum_{j\in [m]} p_j\, \exp(i \pi  \la T^n a_j , \xi\ra ).
\end{equation*}

\subsection{Self-similar measures as convolutions}
\label{subsec:decomposition-large-and-small-part}

Let $\mu=\mu(T,a,p)$. Since, by \eqref{eq:convolution-random-sum}, $\mu$ is the distribution of an absolutely convergent random sum, we can ``split and rearrange'' the series to express $\mu$ as a suitable convolution of two measures. For us, the relevant decomposition is to ``skip every $k$-th term'' and ``keep every $k$-th term'', for a sufficiently large integer $k$. More precisely, using the notation of \eqref{eq:convolution-random-sum}, we define
\begin{align*}
\nu_k &\sim   \sum_{k\nmid n} T^{n-1} A_n,\\
\eta_k &\sim  \sum_{k\mid n} T^{n-1} A_n.
\end{align*}
Both $\nu_k$ and $\eta_k$ are still self-similar measures. Indeed, $\eta_k=\mu(T^k,a,p)$. The explicit expression for $\nu_k$ is more cumbersome to write down, but note that $\nu_k\sim\sum_{\ell=0}^\infty T^{k\ell} B_\ell$, where $B_\ell = \sum_{j=1}^{k-1} T^{j-1} A_{j+k\ell}$. As the $B_\ell$ are  i.i.d.\ Bernoulli random variables, it is now clear that $\nu_k$ is indeed self-similar, and one can easily read off the translations and probabilities. In particular, we have

\begin{lemma} \label{lem:sim-dim-large-part}
Let $\mu,\nu_k$ be as above, and write $s,s_k$ for their respective similarity dimensions. Then $s_k=(1-1/k)s$.
\end{lemma}
\begin{proof}
Let $\mu=\mu(T,a,p)$, and note that $\nu_k=\mu(T^k,a^{(k)},p^{(k)})$, where $p^{(k)}\in \PP_{m^{k-1}}$ and
\[
p^{(k)}=  (p_{i_1}\cdots p_{i_{k-1}})_{(i_1\ldots i_{k-1})\in [m]^{k-1}}.
\]
Hence $h(p^{(k)})= (k-1)h(p)$, and, writing $r$ for the contraction ratio of $T$,
\[
s_k = \frac{(k-1)h(p)}{k\log(1/r)} = \left(1-\frac{1}{k}\right) s.
\]
\end{proof}

\subsection{Projection maps}
\label{subsec:projection-IFS}

Self-similar measures can also be seen as projections of Bernoulli measures on code spaces. More precisely, if $a\in\R^{md}$, then the self-similar measure $\mu(T,a,p)$ is the push-down of the Bernoulli measure $p^\N$ on $[m]^\N$, via the map
\[
\Pi(\iii) = \Pi_{T,a}(\iii) = \sum_{n=1}^\infty T^{n-1}a_{i_n}.
\]
The map $\Pi$ is known as the \emph{projection} or \emph{coding} map for the IFS. Indeed, this is just another way of expressing \eqref{eq:convolution-random-sum}. Projection maps will be repeatedly used in the sequel without further reference.


\section{Variants of the Erd\H{o}s-Kahane argument}
\label{sec:EK}
\subsection{Introduction}

Recall that Erd\H{o}s \cite{Erdos39} proved that Bernoulli convolutions are in $\mathcal{D}_1$ outside of a zero-measure set of exceptions, and Kahane \cite{Kahane71} pointed out that the argument in fact shows that the exceptional set has zero Hausdorff dimension. In this section we prove a number of variants of this, suited to our main results. In each case, the aim is to show that, in a given parametrized family of measures $\{ \mu_u\}_{u\in U}$ on the real line, the set of parameters $u$ such that $\mu_u$ has no power Fourier decay ($\mu_u\notin\mathcal{D}_1)$ has zero Hausdorff dimension. The arguments have many points in common with the original proof of Erd\H{o}s-Kahane as presented in \cite{PSS00}, but each of them has its own peculiarities.

\subsection{Families of self-similar measures with varying translations}

The next result is a key step in the proof of Theorem \ref{mainthm:parametrized}; here the parameter comes in the translations, rather than the contraction ratio. A special case of this was obtained in \cite[Theorem 1.4]{DFW07}; although the proof is very similar, we provide full details for completeness, and because it provides a blueprint for the slightly more involved variants that we will encounter later.

\begin{prop} \label{prop:EK-translations}
Suppose $m\ge 3$. There exists a set $E\subset (0,1)\times \R$ of zero Hausdorff dimension such that $\mu(\lambda,a,p)\in\cD_1$ whenever
 \[
  \left(\lambda,\frac{a_k-a_i}{a_j-a_i}\right) \notin E \quad\text{ for some } a_i<a_j<a_k.
 \]
\end{prop}

The proof of the theorem depends on a combinatorial lemma, which we state and prove first. Here and below, $\|t\|$ denotes the distance from a real number $t$ to the nearest integer.

\begin{lemma} \label{lem:EK-translations}
Fix a compact set $H=[\lam_0,\lam_1]\times [U_0,U_1]\subset (0,1)\times (0,\infty)$. There exists a constant $c_H>0$ such that, for any $N\in\N$ and $\delta\in (0,1/2)$, the set
\begin{equation} \label{eq:covered-set}
\left\{ (\lam,u)\in H : \max_{t\in [1,\lam^{-1}]}\frac1N\left|\left\{ n\in [N] : \max(\| t \lam^{-n}\|,\|tu\lam^{-n}\|) \le c_H \right\}\right| > 1- \delta\right\}
\end{equation}
can be covered by $\exp(O_H(\delta N))$ balls of radius $\lam_1^N$.
\end{lemma}

The lemma says that the set of  parameters $(\lam,u)$ such that, for some $t\in [1,\lam^{-1}]$, both numbers $t\lam^{-n}$ and $tu\lam^{-n}$ are very close to an integer for ``almost all'' values of $n\in [N]$, is very small. It will turn out that the ``bad'' set for the purposes of Proposition \ref{prop:EK-translations} can be controlled by sets of the kind appearing in the lemma. Roughly speaking, each time one of $t\lam^{-n}$ or $tu\lam^{-n}$ is far from an integer, the Fourier transform in question at the frequency $\xi=t \lam^{-N}$ will drop by a constant factor.

\begin{proof}[Proof of Lemma \ref{lem:EK-translations}]
We follow the scheme of \cite[Proposition 6.1]{PSS00}, with suitable variants. All the constants implicit in the $O(\cdot)$ notation are allowed to depend on $H$. During the proof it is always understood that $(\lam,u)\in H$ and $t\in [1,\lam^{-1}]$.

Set $\theta=\lambda^{-1}$. For $n\in [N]$, write
\begin{align*}
t\theta^n &= K_n + \e_n,\\
t u\theta^n &= L_n + \delta_n,
\end{align*}
where $K_n,L_n$ are integers and $\e_n,\delta_n$ are in $[-1/2,1/2)$. All these numbers depend on $t, u$ and $\theta$. Note that
\begin{align*}
u&\in B(L_N/K_N, O(1/K_N)) \subset B(L_N/K_N, O(\lam_1^N)),\\
\theta&\in B(L_{N}/L_{N-1}, O(1/L_N)) \subset B(L_N/L_{N-1}, O(\lam_1^N)).
\end{align*}
Hence we need to estimate the number of possible sequences $(K_i,L_i)_{i\in [N]}$. A calculation, for which the reader is referred to \cite[Lemma 6.3]{PSS00}, shows that
\begin{align*}
| K_{n+2} - K_{n+1}^2/K_n | &= O(\max(|\e_n|,|\e_{n+1}|,|\e_{n+2}|)),\\
| L_{n+2} - L_{n+1}^2/L_n | &= O(\max(|\delta_n|,|\delta_{n+1}|,|\delta_{n+2}|)).
\end{align*}
\begin{sloppypar}
\noindent This shows that:
\begin{enumerate}[(i)]
\item Given $(K_n,L_n), (K_{n+1},L_{n+1})$, there are at most $O(1)$ possible values for $(K_{n+2},L_{n+2})$, uniformly in $t,u,\theta$. There are also $O(1)$ possible values for $(K_1,L_1), (K_2,L_2)$.
\item There is a constant $c_H>0$, such that if $\max(|\e_i|,|\delta_i|) < c_H$ for $i=n,n+1,n+2$, then $(K_n,L_n), (K_{n+1},L_{n+1})$ uniquely determine $(K_{n+2},L_{n+2})$, again independently of $t,u,\theta$.
\end{enumerate}
\end{sloppypar}
\noindent We claim that, for each fixed set $A\subset [N]$ with $|A|\ge (1-\delta)N$, the set
\[
\left\{ (K_n,L_n)_{n\in [N]}:\max(\| t \theta^{n}\|,\|tu\theta^{n}\|) \le c_H \text{ for some } \lam,u,t \text{ and all }n\in A \right\}
\]
has  cardinality $\exp(O(\delta N))$. Indeed, fix such an $A$ and let $\widetilde{A}=\{ i\in [3,N]: i-2, i-1, i\in A\}$; then $|\widetilde{A}|\ge (1-3\delta)N-3$. If we set
\[
\Lambda_j = (K_i,L_i)_{i\in [j]},
\]
then (i), (ii) above show that $|\Lambda_{j+1}|=|\Lambda_j|$ if $j\in \widetilde{A}$ and $|\Lambda_{j+1}|=O(|\Lambda_j|)$ otherwise. Hence $|\Lambda_N| \le O(1)^{3\delta N}$, as claimed.

Since the number of subsets $A$ of $[N]$ of size $j\ge (1-\delta)N$ is also bounded by $\exp(O(\delta N))$ (using e.g. Stirling's formula), we conclude that there are $\exp(O(\delta N))$ pairs $(K_N,L_N)$, and also $\exp(O(\delta N))$ pairs $(L_N,L_{N-1})$, such that $\max(|\e_n|,|\delta_n|) < 1/c_H$ for at least $(1-\delta)N$ values of $n\in [N]$. Hence the set \eqref{eq:covered-set} can be covered by $\exp(O(\delta N))$ balls of radius $\lam_1^N$, and this finishes the proof.
\end{proof}

We can now conclude the proof of the proposition.

\begin{proof}[Proof of Proposition \ref{prop:EK-translations}]
Let $H_M = [1/M,1-1/M]\times [1/M,M]$. Let $c_M=c_{H_M}$ be the constant given by Lemma \ref{lem:EK-translations} for this set, and denote the set given in \eqref{eq:covered-set} with $\delta=1/\ell$ by $E_{M,\ell,N}$. Define
\[
 E = \bigcup_M \bigcap_\ell \limsup_N E_{M,\ell,N}.
\]
Let us first show that $\hdim E=0$. For this it is enough to show that, for any fixed $M$,
\begin{equation} \label{eq:dim_0_in_limit}
\lim_{\ell\to\infty} \hdim \left(\limsup_N E_{M,\ell,N}\right) =0.
\end{equation}
By Lemma \ref{lem:EK-translations}, $E_{M,\ell,N}$ can be covered by $\exp(O(N/\ell))$ balls of radius $(1-1/M)^N$. Hence
\[
\cH^s\left(\bigcup_{N=N_0}^\infty  E_{M,\ell,N} \right) \le \sum_{N=N_0}^\infty \exp(O(N/\ell))(1-1/M)^{sN} \le \exp(-\Omega(N_0)),
\]
provided $s>O_M(1/\ell)$. (Recall that $Y=\Omega(N_0)$ means that $Y\ge c\,N_0$ for some constant $c>0$.) Thus $\hdim \left(\limsup_N E_{M,\ell,N}\right)=O_M(1/\ell)$, showing that \eqref{eq:dim_0_in_limit} holds, and so that $\hdim(E)=0$.

Now suppose $(\lambda,a)$ is such that $(\lambda,u)\notin E$ where $u=(a_k-a_i)/(a_j-a_i)$ and $a_i<a_j<a_k$. Since the $H_M$ exhaust $(0,1)\times (0,\infty)$, we can fix $M$ such that $u\in H_M$. The definition of $E$ then implies that there are $\ell, N_0\in\N$ such that, for any $N\ge N_0$,
\begin{equation} \label{eq:EK-translations-comb}
  \max_{t\in [1,\theta]}\frac1N\left|\left\{ n\in [N] : \max(\| t \lam^{-n}\|,\|tu\lam^{-n}\|) \le c_M \right\}\right| < 1- 1/\ell.
\end{equation}

Since $\mu(\lambda,a,p)$ is homothetic to $\mu(\lambda,b,p)$ where $b_j=(a_j-a_0)/(a_1-a_0)$, we may assume, after relabeling and without loss of generality, that $a_0=0, a_1=1$ and $a_2=u$.

Let $\mu=\mu(\lambda,b,p)$. Write $\xi=\lambda^{-N}t$ with $1\le t\le \lambda^{-1}$. Then, using the expression \eqref{eq:FT-ssm} for the Fourier transform of a self-similar measure,
\[
 \left| \wh{\mu}(\xi) \right| \le \prod_{n=0}^{N-1} \left| \sum_{j\in [m]} p_j\, \exp(i\pi \lambda^{n-N} t b_j )\right|.
\]
Since $tb_0=0, t b_1=t$ and $t b_2=tu$, we conclude from \eqref{eq:EK-translations-comb} that there is $\rho=\rho(c_M,p)>0$ such that
\[
  \left| \wh{\mu}(\xi) \right| \le \exp(-\rho N/\ell) = O(|\xi|^{-\rho|\log\lam|/\ell}).
\]
\end{proof}

\subsection{Projections of self-similar measures}

We next establish another result of Erd\H{o}s-Kahane type, which will be needed to establish Theorem \ref{mainthm:projection-measures}.

\begin{prop} \label{prop:EK-projections}
Fix $\lambda\in \cS_2\setminus\R$, $m\ge 2$ and $a\in\R^{2m}$ which is not of the form $a=(v,\ldots,v)$ for $v\in\R^2$. There exists a set $E\subset [0,\pi)$ of Hausdorff dimension $0$, such that for all $\beta\in [0,\pi)\setminus E$ and all $p\in\PP_m,$
\[
P_\beta\mu(\lambda,a,p)\in \cD_1,
\]
where $P_\beta$ is the projection onto a line making angle $\beta$ with the $x$-axis.
\end{prop}

As in Proposition \ref{prop:EK-translations}, we first establish a combinatorial lemma.
\begin{lemma}  \label{lem:EK-projections}
Let $\theta>1, \alpha\in (0,2\pi)$, $\alpha\ne \pi$, be fixed.  There is a constant $c_{\theta,\alpha}>0$ such that the set
\begin{equation} \label{eq:EK-projections-discrete-bad-set}
 \left\{ \beta\in [0,2\pi) : \max_{|t|\in [1,\theta]}\frac1N\left|\left\{ n\in [N] : \| t \theta^n \cos(\beta+n\alpha) \| \le c_{\theta,\alpha} \right\}\right| > 1- \delta\right\}
\end{equation}
can be covered by $\exp(O(\delta N))$ balls of radius $\theta^{-N}$.
\end{lemma}
\begin{proof}
The constants implicit in the $O$-notation are allowed to depend on $\theta$ and $\alpha$.
All the calculations involving angles below are understood to be modulo $2\pi$. We write
\[
t \theta^n \cos(\beta+n\alpha) = K_n+\e_n,  \quad K_n\in\Z,\e_n\in [-1/2,1/2).
\]
For any $\gamma\in\R$, write $\omega_\gamma=(\cos\gamma,\sin\gamma)$, and notice that
\begin{equation} \label{eq:EK-projections-key-relation}
t\theta^{n+j} \la\omega_{\beta+n\alpha},\omega_{-j\alpha}\ra = K_{n+j}+\e_{n+j}.
\end{equation}
From this it follows that
\[
\cos(\alpha)-\sin(\alpha)\tan(\beta+n\alpha) = \frac{\la\omega_{\beta+n\alpha},\omega_{-\alpha}\ra}{\la\omega_{\beta+n\alpha},\omega_{0}\ra}= \frac{K_{n+1}+\e_{n+1}}{\theta (K_n+\e_n)}\,,
\]
whence
\[
\beta+n\alpha = \arctan\left(\cot(\alpha) - \frac{K_{n+1}+\e_{n+1}}{\theta (K_n+\e_n)\sin(\alpha)}\right)=: f\left(\frac{K_{n+1}+\e_{n+1}}{K_{n}+\e_{n}}\right).
\]
Here the function $f$ is $C^1$ with bounded derivative.

Since $\alpha\neq 0,\pi$, it follows from \eqref{eq:EK-projections-key-relation} that $\max(K_{N-1},K_N)=\Omega(\theta^N)$. Thus it follows from the above that
\begin{equation} \label{eq:EK-projections-covering-balls}
\beta \in B\left(-(N-j)\alpha+f(K_{N-j+1}/K_{N-j}),O(\theta^{-N})\right)\quad\text{for some }j\in\{0,1\}.
\end{equation}

Let $C_1,C_2$ be real constants (depending on $\alpha$) such that $\omega_{-2\alpha}=C_1 \omega_0 + C_2\omega_{-\alpha}$. Then,  using \eqref{eq:EK-projections-key-relation} again,
\[
K_{n+2}+\e_{n+2} = C_1 \theta^2 (K_n+\e_n) + C_2 \theta (K_{n+1}+\e_{n+1}),
\]
whence
\[
K_{n+2} = C_1 \theta^2 K_n + C_2\theta K_{n+1} + O(\max|\e_n|,|\e_{n+1}|,|\e_{n+2}|).
\]
Following the scheme of the proof of Lemma \ref{lem:EK-translations}, we obtain that there exists $c_{\theta,\alpha}>0$ such that there are $\exp(O(\delta N))$ possible sequences $(K_n)_{n\in [N]}$ satisfying $|\e_n|<c_{\theta,\alpha}$ for at least $(1-\delta)N$ values of $n\in [N]$. Together with \eqref{eq:EK-projections-covering-balls} this yields the claim.
\end{proof}

\begin{proof}[Proof of Proposition \ref{prop:EK-projections}]
Fix $\lambda, a, p$ as in the statement of the proposition, and write $\mu=\mu(\lambda,a,p)$. Let $\lambda=\theta^{-1}\omega_\alpha$, where $\theta>1$ and $\omega_\alpha=(\cos\alpha,\sin\alpha)$. After applying a similarity to $\mu$, which does not affect the statement (it does rotate the exceptional set of $\beta$), we may and do assume that $a_1=(0,0)$ and $a_2=(1,0)$. Let
\[
E = \bigcap_{\ell\in \N}\limsup_{N\to\infty} E_{\ell,N},
\]
where
\[
E_{\ell,N}= \left\{ \beta: \max_{|t|\in [1,\theta]}\frac1N\left|\left\{ n\in [N] : \| t \theta^n \cos(\beta-(N-n)\alpha) \| \le c_{\theta,\alpha} \right\}\right| > 1- \frac{1}{\ell}\right\},
\]
and $c_{\theta,\alpha}>0$ is the constant from Lemma \ref{lem:EK-projections}. As in the proof of Proposition \ref{prop:EK-translations} above, it follows from Lemma \ref{lem:EK-projections} that $\hdim E=0$. Thus the task is to show that if $\beta\in [0,2\pi)\setminus E$, then $P_\beta\mu\in\cD_1$. Assume then that $\beta\notin E$; this means that there are $\ell,N_0\in\N$ such that $\beta\notin E_{\ell,N}$ for all $N\ge N_0$.

Fix a frequency $\xi= t\theta^N$ where $|t|\in [1,\theta]$. Denote by $R_\alpha$ the rotation through the angle $\alpha$, so that
$\lam a_j = R_\alpha \theta^{-1} a_j$. Then, since the Fourier transform of a projection is the restriction of the Fourier transform to the line with the corresponding slope, (\ref{eq:FT-ssm}) implies:
\begin{align*}
\wh{P_\beta\mu}(\xi)&=\wh{\mu}(\xi\omega_\beta) \\
&\le \prod_{n=0}^{N-1} \left| \sum_{j\in [m]} p_j\,\exp(i \pi  \la \theta^{-n} R_{n\alpha} a_j , t\theta^N \omega_\beta\ra ) \right|\\
&= \prod_{n=0}^{N-1} \left| \sum_{j\in [m]} p_j\,\exp(i \pi t \theta^{N-n} \la a_j  ,  \omega_{\beta-n\alpha}\ra )\right|.
\end{align*}
Since $\la a_1,\omega_\gamma\ra =0$ and $\la a_2,\omega_\gamma\ra=\cos\gamma$, there is $\rho=\rho(\lam,a,p)>0$ such that
\[
 \left| \sum_{j\in [m]} p_j\,\exp\left(i \pi t \theta^{N-n} \la a_j  ,  \omega_{\beta-n\alpha}\ra \right)\right| < 1-\rho
\]
whenever $\|t\theta^{N-n}\cos(\beta-n\alpha)\|>c_{\theta,\alpha}$. The desired conclusion then follows from the definition of $E_{\ell,N}$.
\end{proof}

\subsection{Convolutions of self-similar measures}

We conclude with yet another variant of Erd\H{o}s-Kahane, this time suited to study convolutions of self-similar measures; it will be a key part of the proof of Theorem \ref{mainthm:convolutions}.

\begin{prop} \label{prop:EK-convolutions}
Given $\lambda_1,\lambda_2\in (0,1)$, there exists a set $E\subset\R$ of zero Hausdorff dimension, such that the following holds: if $u\in\R\setminus E$, then for any $m_1, m_2\ge 2$, any $a_i\in\R^{m_i}$, and any $p_i\in\PP_{m_i}$, $i=1,2$,
\[
\mu(\lambda_1,a_1,p_1) * T_u \mu(\lambda_2,a_2,p_2)\in \cD_1,
\]
where $T_u(x)=ux$.
\end{prop}

As before, this depends on the following combinatorial statement.
\begin{lemma} \label{lem:EK-convolutions}
Fix $\theta_2>\theta_1>1$, and write $k(n)=\min\{ k: \theta_2^k <\theta_1^n\}$. Also let $H=[-M,-1/M]\cup [1/M,M]$. Then there exists a constant $c_H>0$ such that
\begin{equation} \label{eq:EK-convolutions-discrete-bad-set}
 \left\{ u\in H : \max_{|t|\in [1,\theta_1]}\frac1N\left|\left\{ n\in [N] : \max(\| t \theta_1^n\|,\|tu\theta_2^{k(n)}\|) \le c_H \right\}\right| > 1- \delta\right\}
\end{equation}
can be covered by $\exp(O_H(\delta N))$ intervals of length $\theta_1^{-N}$.
\end{lemma}
\begin{proof}
All constants implicit in the $O$ notation may depend on $H$. For $n\in [N]$, write
\begin{align*}
t\theta_1^n &= K_n + \e_n,\\
t u\theta_2^{k(n)} &= L_n + \delta_n,
\end{align*}
where $K_n,L_n$ are integers and $\e_n,\delta_n$ are in $[-1/2,1/2)$; these numbers depend on $t, u$. Write $r_n=\theta_1^n/\theta_2^{k(n)}$; note that $r_n=\Theta(1)$, and therefore
\begin{equation} \label{eq:EK-convolutions-covering-balls}
u \in B(r_n L_n/K_n, O(1/K_n)) \subset B\left(r_n L_n/K_n, O(\theta_1^{-n})\right).
\end{equation}
We estimate the number of possible sequences $(K_i,L_i)_{i\in [N]}$. We bound this crudely by the product of the number of sequences $(K_i)_{i\in [N]}$ and the number of sequences $(L_i)_{i\in N}$. For $K_i$, note that
\[
K_{n+1}-\theta_1 K_n = O(\max(|\e_n|,|\e_{n+1}|)),
\]
and therefore
\begin{enumerate}[(i)]
\item Given $K_n$, there are at most $O(1)$ possible values for $K_{n+1}$, uniformly in $t,u$. (There are also $O(1)$ possible values for $K_1$.)
\item There is a constant $c_H>0$, such that if $\max(|\e_n|,|\e_{n+1}|) < c_H$, then $K_n$ uniquely determine $K_{n+1}$, again independently of $t,u$.
\end{enumerate}
An argument analogous to that in the proof of Lemma \ref{lem:EK-translations} shows that there are $\exp(O(\delta N))$ possible values for $(K_n)_{n\in [N]}$.

Regarding the sequence $L_n$, note that $L_{n+1}+\delta_{n+1} \neq L_n+\delta_n$ only when $k(n+1)=k(n)+1$. Hence, the same analysis restricted to the set $\{ n\in N: k(n+1)=k(n)+1\}$ (which obviously has fewer than $N$ elements) yields also that  there are $\exp(O(\delta N))$ values for $L_N$ (here we use that $u$ lies in a compact set).  Thus there are $\exp(O(\delta N))$ values for $r_N L_N/K_N$, and in light of \eqref{eq:EK-convolutions-covering-balls} this completes the proof.

\end{proof}

\begin{proof}[Proof of Proposition \ref{prop:EK-convolutions}]
Let $H_M = [-M,-1/M]\cup [1/M,M]$. Let $c_M=c_{H_M}$ be the constant given by Lemma \ref{lem:EK-convolutions} for this set, and denote the set given in \eqref{eq:EK-convolutions-discrete-bad-set} with  $\theta_i=\lam_i^{-1}$ and $\delta=1/\ell$ by $F_{M,\ell,N}$. Further, let $E_{M,\ell,N}=\{ u/r_N:u\in F_{M,\ell,N}\}$, where $r_N=\lambda_1^{-N}\lambda_2^{k(N)}$ and $k(N)$ is the smallest integer $k$ such that  $\lambda_1^{-N}\lambda_2^k<1$. Define
\[
 E = \{0\} \cup \bigcup_{M>0} \bigcap_{\ell\in \N} \limsup_N E_{M,\ell,N}.
\]
An analysis nearly identical to that in the proof of Proposition \ref{prop:EK-translations} shows that $\hdim E=0$.

Suppose $u\in\R\setminus E$. Then $u\in H_M$ for some $M$, and there are $\ell,N_0\in\N$ such that, for all $N\ge N_0$,
\[
 \min_{|t|\in [1,\lam_1^{-1}]}\frac1N\left|\left\{ n\in [N] : \max(\| t \lam_1^{-n}\|,\|tu r_N\lam_2^{-k(n)}\|) > c_M \right\}\right| \ge \frac{1}{\ell}.
\]
Thus, one of the following two alternatives occur for each $t$ with $|t|\in [1,\lam_1^{-1}]$ and $N\ge N_0$:
\begin{align*}
\text{(i) } & \frac1N\left|\left\{ n\in [N] : \| t \lam_1^{-n}\| > c_M \right\}\right| \ge \frac{1}{2\ell},\\
\text{(ii) } & \frac1N\left|\left\{ n\in [N] : \| tu r_N\lam_2^{-k(n)}\| > c_M \right\}\right| \ge \frac{1}{2\ell}.
\end{align*}

Pick a frequency $\xi = t\lambda_1^{-N}$ where $|t|\in [1,\lambda_1^{-1}]$. Then also $\xi = t r_N \lambda_2^{-k(N)}$. Thus, denoting $\nu=\mu_1*T_u\mu_2$, where $\mu_1=\mu(\lambda_1,a_1,p_1)$ and $\mu_2= \mu(\lambda_2,a_2,p_2)$ we have, using the convolution formula,
\[
\wh{\nu}(\xi) = \wh{\mu}_1(t \lam_1^{-N}) \wh{\mu}_2(t u r_N  \lam_2^{-k(N)}).
\]
If the alternative (i) above holds, then a standard argument (similar to that of Proposition \ref{prop:EK-translations}) shows that
\[
|\wh{\mu}_1(t \lam_1^{-N})| = O(|\xi|^{-\gamma})
\]
for some $\gamma>0$ independent of $t$ and $N$ and hence of $\xi$. Otherwise, if alternative (ii) holds then, using the fact that $k(n)$ stays constant for at most $O(1)$ values of $n$ before increasing by $1$, we analogously get that
\[
\left|\wh{\mu}_2(t u r_N  \lam_2^{-k(N)})\right| = O(|\xi|^{-\gamma}),
\]
for some $\gamma>0$. Since the Fourier transform of a probability measure is bounded by $1$, in either case we obtain $|\wh{\nu}(\xi)|=O(|\xi|^{-\gamma})$, as desired.
\end{proof}

\section{A criterion for the absolute continuity of a convolution}
\label{sec:abs-cont-of-convolution}

In this section we establish a new criterion to guarantee that the convolution of two measures is absolutely continuous with an $L^q$ density. Before stating it, we review several concepts of dimension of a measure.

For $q>0, q\neq 1$, the (lower) $L^q$ dimension of a measure $\mu\in\cP_d$ of compact support is defined as
\[
D_q(\mu) = \liminf_{n\to\infty} \frac{\log S_{n,q}\mu}{(1-q)n},
\]
where
\[
S_{n,q}\mu = \sum_{I\in\mathcal{D}_n} \mu(I)^q,
\]
with $\mathcal{D}_n$ the partition into dyadic cubes of side length $2^{-n}$. Recall that logarithms are to base $2$. There are several equivalent formulations of this definition; we will require the following one.

\begin{lemma} \label{lem:equivalence-Lq-dim}
For any $\mu\in\cP_d$ and $q>1$,
\[
D_q(\mu) = \liminf_{r\downarrow 0} \frac{\log\left(r^{-d}\int \mu(B(x,r))^q \, dx\right)}{(q-1)\log r}.
\]
Moreover, the $\liminf$ may be taken along any sequence $r_n\downarrow 0$ with $\log r_{n+1}/\log r_n\to 1$.
\end{lemma}
\begin{proof}
This is standard but we include the proof as we have not been able to trace it in the literature. Thanks to the logarithms, it is enough to prove it for the sequence $r_n=2^{-n}$ and interpolate. Let $\ell$ be such that $2^\ell > \sqrt{d}$. Writing $I_{n,x}$ for the dyadic cube of side length $2^{-n}$ that contains $x$, we have
\[
S_{n+\ell,q}\mu = 2^{(n+\ell)d}\int \mu(I_{n+\ell,x})^q \,dx \le O_d(1) 2^{nd}\int \mu(B(x,2^{-n}))^q \,dx.
\]
On the other hand, $\mu(B(x,2^{-n}))\le \sum_{I\in \mathcal{D}_{n,x}} \mu(I)$, where $\mathcal{D}_{n,x}$ are the cubes in $\mathcal{D}_n$ which intersect $B(x,2^{-n})$. Since there are $O_d(1)$ such cubes, and each cube in $\mathcal{D}_n$ is in $\mathcal{D}_{n,x}$ for a set of $x$ of Lebesgue measure $\Theta(2^{-nd}$), we have
\begin{align*}
2^{nd}\int \mu(B(x,2^{-n}))^q \,dx &\le 2^{nd} \int \left(\sum_{I\in \mathcal{D}_{n,x}} \mu(I)\right)^q \,dx \\
&\le O_{d,q}(1) 2^{nd} \int \sum_{I\in \mathcal{D}_{n,x}} \mu(I)^q \,dx \\
&\le O_{d,q}(1) \sum_{I\in\mathcal{D}_n} \mu(I)^q = O_{d,q}(1) S_{n,q}\mu.
\end{align*}
But $S_{n+\ell,q}\mu = \Theta_\ell(S_{n,q}\mu)$ (see e.g. \cite[Lemma 2.2]{PeresSolomyak00}), hence combining both inequalities yields the lemma.
\end{proof}
We will also need the concept of \emph{Fourier} dimension. This is usually defined for sets, but here we need it for measures.
\begin{defn}
For any $\mu\in\cP_d$, its \emph{Fourier dimension} is
\[
\fdim\mu = \sup\{ \sigma\ge 0: \wh{\mu}(\xi) = O_{\mu,\sigma}(|\xi|^{-\sigma/2})\}.
\]
\end{defn}

The $1/2$ factor arises because of the connection between Fourier decay and Hausdorff dimension: for any $\mu\in\cP_d$ it holds that $\fdim\mu\le \hdim\mu$, see e.g. \cite[Section 12.17]{Mattila95}. The inequality may be (and often is) strict for self-similar measures. For example, any self-similar measure (indeed, \emph{any measure}) on the middle-thirds Cantor set has Fourier dimension zero. We recall that measures with $\fdim\mu=\hdim\mu$ are called \emph{Salem} measures, but we will not require this concept here.

The $L^2$ dimension plays a distinguished r\^{o}le due to its connection with energy and $L^2$ norms, and is also known as \emph{correlation} dimension. The following was proved in \cite[Lemma 2.1]{Shmerkin13}.

\begin{lemma} \label{lem:convolution-AC}
Let $\mu,\nu\in\cP_d$.
\begin{enumerate}[(i)]
 \item If $\hdim\mu=d$ and $\fdim\nu>0$, then $\mu*\nu\ll\cL_d$.
 \item If $D_2\mu=d$ and $\fdim\nu>0$, then $\mu*\nu$ has a density in $L^2$, and it also has fractional derivatives of some order in $L^2$.
\end{enumerate}
\end{lemma}

In light of this result, it seems natural to ask if a similar result holds if $D_p\mu=d$ for $p\neq 2$. Although the proof of \cite[Lemma 2.1]{Shmerkin13} does not seem to extend to  other values of $p$, we are able to show that the answer is affirmative:

\begin{thm} \label{thm:abs-cont-convolution}
Let $1<p<\infty$. Let $\mu,\nu\in\cP_d$, with $\mu$ of compact support. If
\begin{align*}
d - D_p\mu< \fdim(\nu) & \quad \text{if } p\in (1,2],\\
(p-1)(d-D_p\mu)<\fdim(\nu) & \quad \text{if } p\in (2,+\infty),
\end{align*}
then $\mu*\nu$ is absolutely continuous with an $L^p$ density. In particular, this is the case if $D_p\mu=d$ and $\fdim\nu>0$.
\end{thm}
\begin{proof}
We use a Littlewood-Paley decomposition: let $\psi\in C_0^\infty(\R^d)$ be a function supported on the annulus $\{ 1/2 \le |x| \le 2\}$ with $0\le\psi\le 1$ such that $\sum_{k=0}^\infty \psi_k(x)=1$, with $\psi_k(x)=\psi(2^{-k}x)$ for $k\ge 1$, and $\psi_0\in C_0^\infty(\R^d)$ is supported on $\{ |x|\le 2\}$. See e.g. \cite[Lemma 4.1]{PeresSchlag00} for its construction (we remark that often one defines $\psi_k = \psi(2^{-k}\cdot)$ for all $k\in\Z$; our $\psi_0$ corresponds to adding $\psi(2^{-k}\cdot)$ over all non-positive $k$ and setting $\psi_0(0)=1$).

Let $\Delta_k$ be the multiplier operator with symbol $\psi_k$: $\wh{\Delta_k \eta}=\psi_k \wh{\eta}$; hence $\Delta_k \eta$ is the localization of $\eta$ to frequencies of size $\sim 2^k$. It is enough to prove that there is $\delta>0$ such that
\begin{equation} \label{eq:Lp-LittlewoodPaley}
\|\Delta_k(\mu*\nu)\|_p < O(1) 2^{-\delta k} \quad\text{for all }k\in\N.
\end{equation}
Indeed, $(\mu*\nu)_k =\sum_{i=0}^k \Delta_i(\mu*\nu)$ converges to $\mu*\nu$ weakly, and (\ref{eq:Lp-LittlewoodPaley}) shows that it is
an $L^p$-Cauchy sequence, hence it must converge to $\mu*\nu$ in $L^p$.

Any constants implicit in the $O$ notation below may depend on $p$, but are independent of $k,\sigma$ and $D_p\mu$.

Consider the operator $T_k f= f*\Delta_k\nu$ (in other words, the multiplier with symbol $\psi_k\wh{\nu}$). On the one hand, we have
\begin{align*}
\| T_k f\|_1 &\le \| \Delta_k\nu \|_1 \| f\|_1 \le \| \wc{\psi}_k\|_1 \nu(\R^d) \| f\|_1 = \|\wc{\psi}\|_1\|f\|_1,\\
\| T_k f\|_\infty &\le \| \Delta_k\nu \|_1 \| f\|_\infty \le \| \wc{\psi}_k\|_1 \nu(\R^d) \| f\|_\infty = \|\wc{\psi}\|_1\|f\|_\infty.
\end{align*}
On the other hand, by Plancherel, and using the hypothesis, there exists $\sigma>d-D_p\mu$ (if $1<p\le 2$) or $\sigma>(p-1)(d-D_p\mu)$ (if $p>2$) such that
\[
\| T_k f\|_2 = \| \psi_k \wh{\nu} \wh{f}\|_2 \le O(1)\, 2^{-\sigma k/2} \|f\|_2.
\]
Hence, by Riesz-Thorin,
\begin{align}
\| f*\Delta_k\nu\|_p &\le O(1) 2^{-\sigma k/q}\|f\|_p &\text{if } p\in (1,2];  \label{eq:bound-nu} \\
\| f*\Delta_k\nu\|_p &\le O(1) 2^{-\sigma k/p}\|f\|_p &\text{if } p\in [2,+\infty), \label{eq-bound-nu-ge2}
\end{align}
where $q$ is the conjugate exponent of $p$.

Let $\varphi\in C_0^\infty(\R^d)$ be a function supported on $\{ 1/4\le |x|\le 4\}$ such that $\varphi=1$ on $\supp(\psi)$; set $\varphi_k(x)=\varphi(2^{-k}x)$. Then $\varphi_k\psi_k=\psi_k$, and therefore,
\begin{equation} \label{eq:convolution}
\Delta_k(\mu*\nu) = (\wc\varphi_k*\mu)*\Delta_k\nu.
\end{equation}
Fix $\e>0$ and a sufficiently large $N=N(\e)$; taking $N=(\eps p)^{-1}$ works. Since $\wc\varphi$ is a Schwartz function, and $\wc\varphi_k(x) = 2^{dk}\wc\varphi(2^k x)$, we can estimate
\begin{align}
\| \wc\varphi_k*\mu\|_p^p &\le  2^{pdk} \int \left( \int|\wc\varphi(2^k(x-y))|\, d\mu(y)  \right)^p dx \nonumber\\
&\le 2^{pdk} \int  \left( \int O_\eps(1) \left(1+2^k(x-y)\right)^{-N}\, d\mu(y)  \right)^p dx \nonumber\\
&\le O_\eps(1) 2^{pdk}  \int \mu(B(x,2^{-k(1-\e)}))^p \,dx +  O_\eps(2^{-k}), \label{eq:Lp-estimate-1}
\end{align}
where the last line is obtained by estimating the inner integral in the second line over the domains $B(x,2^{-k(1-\e)})$ and its complement separately.

On the other hand, by Lemma \ref{lem:equivalence-Lq-dim} applied with $r_k=2^{-k(1-\e)}$,
\be
2^{(1-\e)d k}   \int \mu\left(B(x,2^{-k(1-\e)})\right)^p \,dx \le O(1) 2^{-k(1-\e)(p-1)(D_p\mu-\e)}.\label{eq:Lp-estimate-2}
\ee
From \eqref{eq:Lp-estimate-1} and \eqref{eq:Lp-estimate-2}, we deduce
\[
\| \wc\varphi_k*\mu\|_p \le O_\eps(1) 2^{k O(\e)} 2^{k(d-D_p\mu)/q}.
\]
Recalling \eqref{eq:bound-nu}, \eqref{eq-bound-nu-ge2}, \eqref{eq:convolution} and the choice of $\sigma$, we conclude that
\begin{align*}
\|\Delta_k(\mu*\nu)\|_p &\le O_\eps(1) 2^{k(O(\e))} 2^{-\delta k},
\end{align*}
for some $\delta>0$ and arbitrary $\e>0$. This shows that \eqref{eq:Lp-LittlewoodPaley} holds, as desired.
\end{proof}



\section{Continuity of $L^q$ dimensions of self-similar measures}
\label{sec:dimensions-of-ssm}

The $L^q$ dimension is not a priori defined for $q=1$. However,  for fixed $n$,
\begin{equation} \label{eq:discrete-Lq-L1-continuity}
\lim_{q\to 1} \frac{\log S_{n,q}\mu}{(1-q)n} =  \frac{ H_n\mu}{n},
\end{equation}
where
\[
H_n\mu = \sum_{I\in\cD_n} \mu(I)\log(1/\mu(I)).
\]
This suggests defining
\[
D_1\mu = \liminf_{n\to\infty} \frac{ H_n\mu}{n}.
\]
The quantity $D_1\mu$ is known as the \emph{entropy dimension} or \emph{information dimension} of $\mu$. In light of \eqref{eq:discrete-Lq-L1-continuity}, it seems reasonable to ask whether $q\mapsto D_q\mu$ is  continuous at $q=1$. It is not hard to construct counterexamples for general measures. However, for self-similar measures, we can establish continuity from the right.

\begin{thm} \label{thm:Lq-dim-continuity}
Let $\mu$ be a homogeneous self-similar measure on $\R^d$ (recall that this means that $\mu=\mu(T,a,p)$ for some $T\in \cS_d$, $a\in\R^{md}$, $p\in\PP_m$). Then
\[
D_1\mu = \lim_{q\to 1^+} D_q\mu.
\]
\end{thm}

\begin{rem}
The proof can be adapted to show that the result holds for all (not necessarily homogeneous) self-similar measures, and even for
self-conformal measures on $\R^d$ under standard bounded distortion assumptions. To see this, one needs to modify the proof of \cite{PeresSolomyak00} with the sharper estimates we develop in the proof of Proposition \ref{prop:Lq-dim-submultiplicative} below.
\end{rem}

Together with Theorem \ref{thm:abs-cont-convolution}, this result is the key to the proof that the densities are in some $L^q$ space, $q>1$, in the settings of Theorems \ref{mainthm:parametrized}, \ref{mainthm:projection-measures} and \ref{mainthm:convolutions}, but we believe that it may also be of independent interest. We prove Theorem \ref{thm:Lq-dim-continuity} at the end of this section, after establishing two preliminary results.

\begin{lemma} \label{lem:Holder-partitions}
Let $\mu$ be a probability measure on a set $K$. Suppose $\{ A_i\}_{i\in I}$ and $\{B_j\}_{j\in J}$ are two measurable partitions of $K$ such that each $A_i$ is covered by at most $M$ of the sets $B_j$. Then, for any $q\ge 1$,
\[
\sum_{i\in I} \mu(A_i)^q \le M^{q-1} \sum_{j\in J}\mu(B_j)^q.
\]
\end{lemma}
\begin{proof}
Using H\"{o}lder's inequality and the inequality $\sum_i p_i^q\le (\sum_i p_i)^q$ for $q\ge 1$, we estimate
\begin{align*}
\sum_{i\in I} \mu(A_i)^q &= \sum_{i\in I} \left( \sum_{j\in J} \mu(A_i\cap B_j) \right)^q\\
&\le \sum_{i\in I} M^{q-1} \sum_{j\in J} \mu(A_i\cap B_j)^q\\
&= M^{q-1} \sum_{j\in J}\sum_{i\in I} \mu(A_i\cap B_j)^q\\
&\le M^{q-1} \sum_{j\in J} \mu(B_j)^q.
\end{align*}
\end{proof}

The above lemma will be used repeatedly in the proof of the following sub-multiplicativity result for the $L^q$ dimension.
\begin{prop} \label{prop:Lq-dim-submultiplicative}
Let $\mu=\mu(T,a,p)$ for some $T\in \cS_d$, $a\in\R^{md}$, $p\in\PP_m$; i.e. $\mu$ is a homogeneous self-similar measure on $\R^d$.
 Then there exists $M>1$ such that for all $q>1$,
\[
S_{n+m,q}\mu \le M^{q-1} \, S_{n,q}\mu \, S_{m,q}\mu.
\]
\end{prop}

We remark that in the course of the proof of \cite[Theorem 1.1]{PeresSolomyak00}, it is shown that (for a more general class of measures) $S_{n+m,q}\mu \le C_q\, S_{n,q}\mu\, S_{m,q}\mu$, for some constant $C_q>1$. Although $C_q$ is not made explicit, one can check that there is $C>1$ such that $C_q\ge C$ for all $q$ and  (as will become apparent later) this is insufficient for the application to Theorem \ref{thm:Lq-dim-continuity}. Hence, although we use the general scheme from \cite{PeresSolomyak00}, we also need to introduce new, more efficient bounds. Roughly speaking, we need to work with partitions at all times, rather than just efficient coverings.

\begin{proof}[Proof of Proposition \ref{prop:Lq-dim-submultiplicative}]
Let $\lam\in (0,1)$ be the contraction ratio of $T$.
After iterating the IFS if needed we can assume $\lam<1/2$. Fix $n,m\in\N$, and let $h$ be the integer such that
\[
\lam^{h+1} \le 2^{-n} < \lam^h.
\]
It will be more convenient to work in the code space. We will denote the Bernoulli measure $p^\N$ on $[m]^\N$ by $\eta$, and the coding map by $\Pi$. Hence, $\mu=\Pi\eta$, where $\Pi(\iii)=\sum_{n=1}^\infty T^{n-1} a_{i_n}$.
We want to replace Euclidean cubes $Q\in \cD_n$ by symbolic analogues $Q'$ which are unions of cylinders of comparable size, and which
partition the code space. So we need to assign some cylinders to each dyadic cube $Q$ of size $2^{-n}$ so that every cylinder is assigned to one and only one cube $Q$. To this end, we label all cubes in $\cD_n$ by their coordinates and consider the induced lexicographic order $\prec$.

For each $Q\in\cD_n$, we consider $Q'\subset [m]^\N$ defined as follows:
\[
Q' = \bigcup\{ [\iii]: \iii\in\Sigma_Q\},
\]
where
\[
\Sigma_Q = \{ \iii\in [m]^h, \Pi[\iii]\cap Q \neq\varnothing \text{ and } \Pi[\iii]\cap R =\varnothing \text{ if  } R\in \cD_n,\ R \prec Q \}.
\]
Recall that for $\iii\in [m]^h$, the symbolic cylinder $[\iii]$ consists of all infinite words in $[m]^\N$ that start with $\iii$. Notice that:
\begin{enumerate}
\item \label{it:Q'-is-partition} $\{ Q':Q\in\cD_n \}$ is a partition of $[m]^\N$;
\item \label{it:Q'-intersects-bounded-cubes} each $\Pi Q'$ intersects at most $O(1)$ cubes in $\cD_n$, and vice versa;
\item \label{it:Q'-contains-full-cylinders} For each $\iii\in [m]^h$ and each $Q\in\cD_n$, either $[\iii]\subset Q'$ or $[\iii] \cap Q'=\varnothing$.
\end{enumerate}
Properties (1) and (3) are clear by construction. Property (2) holds, since the diameter of each $\Pi[\iii]$, with $\iii\in [m]^h$, is $O(1)2^{-n}$, and
every $\Pi[\iii]$, for $[\iii]\in Q'$, intersects $Q$.
By \eqref{it:Q'-is-partition}, \eqref{it:Q'-intersects-bounded-cubes} and Lemma \ref{lem:Holder-partitions},
\be \label{eq:self-similarity-Lq-1}
\sum_{Q\in\cD_n} \eta(Q')^q \le O(1)^{q-1} \sum_{Q\in\cD_n} \eta(\Pi^{-1}Q)^q = O(1)^{q-1} \, S_{n,q}\mu.
\ee
We also claim that, for a given $\iii\in [m]^h$,
\begin{equation} \label{eq:self-similarity-Lq}
\sum_{R\in\cD_{m+n}}  \eta(\Pi^{-1}R\cap [\iii])^q \le O(1)^{q-1}\, p_\iii^q \, S_{m,q}\mu.
\end{equation}
Indeed, since $\eta$ is Bernoulli, $\eta(A\cap [\iii]) = p_\iii \eta(\sigma^h(A\cap [\iii]))$ for any set $A\subset [m]^\N$. The sets
\[
\{ \sigma^h(\Pi^{-1}R\cap  [\iii]) :R\in\cD_{m+n}\}
\]
form a partition of $[m]^\N$, and each $\sigma^h(\Pi^{-1}R\cap [\iii])$ can be covered by $O(1)$ of the sets $\Pi^{-1}Q$, $Q\in\cD_m$.
To see why the last statement holds, observe that $\Pi(\sigma^h(\Pi^{-1}R\cap [\iii]))\cap Q = f_{\iii}^{-1}R\cap Q$. Since $f_{\iii}$ has
contraction ratio $\lam^h\in (2^{-n},2^{-n+1}]$, the set $f_{\iii}^{-1}R$ has diameter $O(1)2^{-n}$ for $R\in \cD_{m+n}$.
Hence \eqref{eq:self-similarity-Lq} follows from Lemma \ref{lem:Holder-partitions}.

We conclude
\begin{align*}
S_{n+m,q}\mu &= \sum_{R\in\cD_{m+n}} \eta(\Pi^{-1}R)^q \\
&= \sum_{R\in\cD_{m+n}} \left(  \sum_{Q\in\cD_n} \eta(\Pi^{-1}R \cap Q') \right)^q   & \text{By (1)}\\
&\le O(1)^{q-1} \sum_{R\in\cD_{m+n}} \sum_{Q\in\cD_n} \eta(\Pi^{-1}R\cap Q')^q  & \text{By H\"{o}lder and (2)}\\
&=  O(1)^{q-1} \sum_{Q\in\cD_n} \sum_{R\in\cD_{m+n}}  \left(\sum_{\iii\in\Sigma_Q} \eta(\Pi^{-1}R\cap [\iii])\right)^q & \text{By }\eqref{it:Q'-contains-full-cylinders}\\
&\le O(1)^{q-1} \sum_{Q\in\cD_n} \left(\sum_{\iii\in\Sigma_Q}  \left(\sum_{R\in\cD_{m+n}}  \eta(\Pi^{-1}R\cap [\iii])^q\right)^{1/q}\right)^q & \text{Minkowski's ineq.}\\
&\le O(1)^{q-1} \sum_{Q\in\cD_n} \left(\sum_{\iii\in\Sigma_Q} O(1)^{(q-1)/q} p_\iii (S_{m,q}\mu)^{1/q}\right)^q &\text{By } \eqref{eq:self-similarity-Lq}\\
&= O(1)^{q-1} \, S_{m,q}\mu \sum_{Q\in\cD_n} \eta(Q')^q \\
&\le O(1)^{q-1} \, S_{m,q}\mu \,S_{n,q}\mu. & \text{By }\eqref{eq:self-similarity-Lq-1}
\end{align*}
\end{proof}

We can now easily deduce Theorem \ref{thm:Lq-dim-continuity}.

\begin{proof}[Proof of Theorem \ref{thm:Lq-dim-continuity}]
It is known that $D_q\mu \le D_1\mu$ for all $q>1$, see \cite[Theorem 1.4]{FLR02}. Thus we need to establish that $\liminf_{q\to 1^+} D_q\mu \ge D_1\mu$.

Fix $\e>0$ and choose $n\in\N$ such that $1/n<\e$ and $H_n\mu > (1-\e) n D_1\mu$. In view of \eqref{eq:discrete-Lq-L1-continuity}, there is $\delta>0$ such that if  $|q-1|<\delta$, then
\[
\frac{\log S_{n,q}\mu}{(1-q)n} > (1-2\e) D_1\mu.
\]
 On the other hand, by Proposition \ref{prop:Lq-dim-submultiplicative} there is $M>1$ such that, if $q\ge 1$, then the sequence
\[
( \log S_{n,q}\mu+ (q-1)\log M )_{n\in\N}
\]
is sub-additive, whence, for $q\in (1,1+\delta)$,
\begin{align*}
D_q\mu &= \lim_{m\to\infty}\frac{\log S_{m,q}\mu}{(1-q)m} \\
&= \sup_{m\in \N} \frac{\log S_{m,q}\mu+(q-1)\log M}{(1-q)m} \\
&\ge \frac{\log S_{n,q}\mu}{(1-q)n} - \frac{\log M}{n}\\
&\ge (1-2\e) D_1\mu - O(\e).
\end{align*}
Since $\e$ was arbitrary, this completes the proof.
\end{proof}

We record the following consequence of Theorems \ref{thm:abs-cont-convolution} and \ref{thm:Lq-dim-continuity}; this is the statement that will be applied in the proofs of our main results.
\begin{cor} \label{cor:convolution-with-density-in-Lq}
Let $\nu$ be a self-similar measure for a homogeneous IFS on $\R^d$ (that is, $\nu=\mu(T,a,p)$ for appropriate $T,a,p$) with $\hdim\nu=d$, and let $\eta\in\mathcal{D}_d$. Then $\nu*\eta$ is absolutely continuous with a density in $L^q$ for some $q>1$.
\end{cor}
\begin{proof}
We start by recalling that Hausdorff and entropy dimensions coincide for self-similar measures, see \cite{FengHu09} (in fact, the proofs of all results we have stated for the Hausdorff dimension of self-similar measures actually apply to entropy dimension).

By assumption, there is $\sigma>0$ such that $\widehat{\eta}(\xi)=O(|\xi|^{-\sigma/2})$. By Theorem \ref{thm:Lq-dim-continuity}, there is $q\in (1,2]$ such that $D_q(\nu)>1-\sigma$. The conclusion then follows from Theorem \ref{thm:abs-cont-convolution}.
\end{proof}

\section{Proofs of main results}
\label{sec:proofs}

\subsection{Proof of Theorem \ref{mainthm:parametrized}}

Empowered with the tools developed in the previous sections, and the dimension results from \cite{NPS12,HochmanShmerkin12,Hochman13}, we can now give short proofs of our main results. We begin here with Theorem \ref{mainthm:parametrized}. We first state a result that follows from Hochman's work \cite{Hochman13,Hochman14}; notice that it is very close to the statement of Theorem \ref{mainthm:parametrized}, except that the conclusion is about Hausdorff dimension rather than absolute continuity.

\begin{thm} \label{thm:dimension-parametrized}
Let $\{ (\lam_u,a_u) \}_{u\in U}$ be a real-analytic map $\R^\ell\supset U\to \cS_1\times \R^m$ such that the following non-degeneracy condition holds: for any distinct infinite sequences $\iii,\jjj\in [m]^\N$, there is a parameter $u$ such that $\Pi_u(\iii)\neq \Pi_u(\jjj)$, where $\Pi_u$ is the projection map corresponding to $(\lam_u,a_u)$.

Then there exists a set $E\subset U$ of Hausdorff (and even packing) dimension at most $\ell-1$, such that for any $u\in U\setminus E$ and any $p\in\PP_m$,
\[
\hdim(\mu(\lam_u,a_u,p)) = \min(s(\lam_u,m,p),1).
\]
\end{thm}
This theorem follows from more general results in \cite{Hochman14}. For $\ell=1$ it is a special case of \cite[Theorem 1.8]{Hochman13}, and the general case follows along similar lines although there are some new technical issues.
We should point that the special case in which $\lam_u$ is constant and only the translations depend on the parameters follows from \cite[Theorem 1.8]{Hochman13} rather easily, see \cite[Prop. 4.3]{FraserShmerkin14} for the special case in which the translations \emph{are} the parameters. This is all that will be needed to derive our Theorems B and D.

We will also need the following simple lemma.
\begin{lemma} \label{lem:analytic-preimage}
Let $f:\R^\ell\supset U\to\R$ be a non-constant real-analytic function. Then $\hdim(f^{-1}E)\le \ell-1+\hdim(E)$ for any set $E\subset\R$.
\end{lemma}
\begin{proof}
Since $f$ is non-constant and real-analytic, its gradient vanishes on a  (possibly empty) countable union $S$ of analytic hypersurfaces. By the implicit function theorem, we can decompose $U\setminus S$ as a countable union $\bigcup_i U_i$, such that for each $i$ there is diffeomorphism $g_i:\R^\ell \supset V_i\to U_i$ that linearizes $f$ in the sense that $f g_i(x_1,\ldots,x_\ell)=x_1$. Since $f^{-1}(E)\subset S\cup \bigcup_i g_i(E\times\R^{\ell-1})$, the lemma follows. (We have used that $\hdim(E\times F)=\hdim E+\hdim F$ if $F$ has equal Hausdorff and box-counting dimensions, see \cite[Product formula 7.3]{Falconer03}.)
\end{proof}

\begin{proof}[Proof of Theorem \ref{mainthm:parametrized}]
Recall from Section \ref{subsec:decomposition-large-and-small-part} that, given $k\in\N$, the self-similar measure $\mu_{u,p}=\mu(\lam_u,a_u,p)$ can be decomposed as
\[
\mu_{u,p} =  \mu(\lam_u^k,a_u^{(k)},p^{(k)}) * \mu(\lam_u^k,a_u,p) =: \nu_{u,p,k} * \eta_{u,p,k},
\]
where $\nu_{u,p,k}$ is obtained by skipping every $k$-th digit, and $\eta_{u,p,k}$ is obtained by keeping only every $k$-th digit. By construction, the projection map $\Pi'_{u,k}$ for $\nu_{u,p,k}$ is related to the projection map $\Pi_u$ for $\mu_{u,p}$ in the following way: if $\iii=(i_1 i_2\ldots)$, with $i_j\in [m]^{k-1}$, then
\[
\Pi'_{u,k}(\iii) = \Pi_u(i_{1,1}\ldots i_{1,k-1}0 i_{2,1}\ldots i_{2,k-1}0\ldots).
\]
Hence the non-degeneracy assumption for the family $\{\mu_{u,p}\}_{u\in U}$ transfers to the family $\{ \nu_{u,p,k}\}_{u\in U}$, and we can apply Theorem \ref{thm:dimension-parametrized} and Lemma \ref{lem:sim-dim-large-part}  to deduce that there exists a set $E'_k\subset U$ of  Hausdorff dimension at most $\ell-1$, such that
\begin{equation} \label{eq:parametrized-large-dim}
\hdim(\nu_{u,p,k}) = \min\left(\left(1-1/k\right)s(\lam_u,m,p),1\right) \quad\text{for all } u\in U\setminus E'_k,\ p\in\PP_m.
\end{equation}
To handle the measures $\eta_{u,p,k}$, the argument differs depending on whether $u\mapsto\lam_u$ is constant or not. In the former case, we must have $m\ge 3$. Indeed, suppose $m=2$. If $|\lam|<1/2$, then $s(\lam,2,p)<1$ for any $p$ and there is nothing to do, while if $|\lam|\ge 1/2$, there are two sequences $\iii,\jjj\in\{0,1\}^\N$ such that $\sum_{n=0}^\infty (i_n-j_n)\lam^{n-1}=0$, and this implies that the non-degeneracy assumption fails for this pair of sequences. Hence we assume that $m\ge 3$ from now on. In this case, the function
\[
f(u)=\frac{a_3(u)-a_1(u)}{a_2(u)-a_1(u)}
\]
is non-constant and real-analytic outside of a countably union of hyper-surfaces of dimension $\ell-1$ (where the denominator vanishes). Otherwise, if either numerator or denominator were constant, the non-degeneracy condition would fail. Now let $F\subset (0,1)\times \R$ be the zero-dimensional exceptional set given by Proposition \ref{prop:EK-translations}, and let $F_k\subset \R$ be the $\lam^k$-slice of $F$. Finally, let $E''_k=f^{-1}(F_k)$. Then $\eta_{u,p,k}\in\mathcal{D}_1$ for all $u\in U\setminus E''_k$ and all $p\in\PP_m$. Moreover, it follows from Lemma \ref{lem:analytic-preimage} that $E''_k$ has Hausdorff dimension at most $\ell-1$.

In the case $\lam_u$ is not constant, we rely on  \cite[Theorem 1.5]{Watanabe12} (see also \cite[Proposition 2.3]{Shmerkin13}) and an analogous argument to also conclude that there is a set $E''_k\subset U$ of Hausdorff dimension at most $\ell-1$, such that $\eta_{u,p,k}\in\mathcal{D}_1$ for all $u\in U\setminus E''_k$ and all $p\in\PP_m$.

Let $E=\bigcup_{k\in\N} E'_k\cup E''_k$; then $\hdim(E)\le \ell-1$. Fix $u\in U\setminus E$ and $p\in\PP_m$ such that $s(\lam_u,m,p)>1$. It follows from \eqref{eq:parametrized-large-dim} that $\hdim \nu_{u,p,k}=1$ if we take $k$ large enough. We now conclude from Corollary \ref{cor:convolution-with-density-in-Lq} that $\mu_{u,p}=\nu_{u,p,k}*\eta_{u,p,k}$ is absolutely continuous with an $L^q$ density for some $q>1$.
\end{proof}

\subsection{Proof of Theorems \ref{mainthm:projection-measures} and \ref{mainthm:projection-sets}}
\label{subsec:proof-projections}

In this section we establish Theorem \ref{mainthm:projection-measures} and deduce Theorem \ref{mainthm:projection-sets} as a corollary. As suggested by the statement, there are two different cases depending on whether the rotational part of the similitude $\lam$ is rational or not. In the irrational case we can say more, thanks to the following result.

\begin{thm} \label{thm:projection-dim-irrational}

Let $\lam = r \exp(2\pi i \alpha)$, where $\alpha/\pi\notin\mathbb{Q}$ and $r\in (0,1)$, and let $m\in\N$, $a\in\R^{2m}$ be such that the measure $\mu=\mu(\lam,a,p)$ satisfies the strong separation condition. Assume further that $r^{q-1}\ge\sum_{i=1}^m p_i^q$ for some $q\in (1,2]$.
Then $D_q (P_\beta\mu)=1$ for all $\beta\in [0,\pi)$.
\end{thm}
\begin{proof}
This can be proved via a similar method to the proof of \cite[Theorem 1.1]{NPS12} (the current setting is in fact somewhat simpler). We sketch the proof, leaving the details to the interested reader. The key is the  inequality
\begin{equation} \label{subm}
S_{k+\ell,q}(P_\beta\mu) \le C_q\, S_{k,q}(P_\beta \mu)\, S_{\ell,q}(P_{\beta+k\alpha}\mu),
\end{equation}
for some constant $C_q>0$,
which implies that
$$
f_k(\beta):= \log S_{k,q}(P_\beta\mu)+\log C_q
$$
is a subadditive cocycle over the $\alpha$ rotation on the circle.
 The proof of (\ref{subm}) is similar to the analogous derivation of \cite[(3.1)]{NPS12} (where $q=2$), but simpler since in our case the irrational rotation arises in a geometrically transparent way. It is also closely related to the proof of Proposition \ref{prop:Lq-dim-submultiplicative} above.

Since $\mu$ satisfies the strong separation condition, $D_q\mu=\frac{\log\sum_{i=1}^m p_i^q}{(q-1)\log r}$. Hence the assumption $r^{q-1}\ge\sum_{i=1}^m p_i^q$ ensures that $D_q\mu\ge 1$.

The conclusion of the proof now follows exactly the argument in \cite{NPS12}. By a version of Marstrand's projection Theorem due to Hunt and Kaloshin \cite{HuntKaloshin97}, $D_q (P_\beta\mu)=1$ for almost all $\beta$ (this is where the hypothesis $q\le 2$ gets used). A result of Furman \cite[Theorem 1]{Furman97} on subadditive cocycles over uniquely ergodic transformations, applied to the cocycle $f_k$, yields that $D_q(P_\beta\mu) \ge 1$ for \emph{all} $\beta$ (this is the step that uses the irrationality of $\alpha/\pi$). Since the upper bound is trivial, this concludes the proof.
\end{proof}

\begin{rem}
 A version of the above theorem for Hausdorff dimension (rather than $L^q$ dimension) was proved for more general self-similar measures in \cite[Theorem 1.6]{HochmanShmerkin12}. The above strengthening is required to obtain the additional information on the densities in Theorem \ref{mainthm:projection-measures}, but if one is only interested in absolute continuity (which is enough to deduce Theorem \ref{mainthm:projection-sets}), then \cite[Theorem 1.6]{HochmanShmerkin12} suffices.
\end{rem}

We can now finish the proof of Theorem \ref{mainthm:projection-measures}.

\begin{proof}[Proof of Theorem \ref{mainthm:projection-measures}]
 Consider first the rational rotation case, $\alpha/\pi\in\mathbb{Q}$. By iterating the IFS, which does not affect the assumptions, we may and do assume that there is no rotation, that is, $\lam\in (0,1)$. We fix the vector of translations $a=(a_1,\ldots,a_m)$.

 Write $\mu_p=\mu(\lam,a,p)$. Then $P_\beta\mu_p = \mu(\lam,P_\beta a,p)$, where $P_\beta a=(P_\beta a_1,\ldots, P_\beta a_m)$. In particular, $P_\beta\mu_p$ is a self-similar measure of similarity dimension  $\hdim\mu_p$. Moreover, the family $(P_\beta\mu_p)_{\beta\in [0,\pi)}$ satisfies the non-degeneracy assumption of Theorem \ref{thm:dimension-parametrized}, since $\Pi_\beta = P_\beta\circ \Pi$, where $\Pi,\Pi_\beta$ are the projection maps for $\mu_p,P_\beta\mu_p$, and $\Pi$ is injective since the IFS generating $\mu_p$ satisfies the strong separation condition. The first claim, in the rational rotation case, then follows from Theorem \ref{mainthm:parametrized}.

 We now turn to the irrational rotation case, $\alpha/\pi\notin\mathbb{Q}$. For each $k\in\N$, decompose $\mu_p=\mu(\lam,a,p)=\nu_{p,k}*\eta_{p,k}$, where $\nu_{p,k}$ corresponds to skipping every $k$-th digit, and $\eta_{p,k}$ to keeping every $k$-th digit. An argument similar to that in Lemma \ref{lem:sim-dim-large-part} yields (using \eqref{eq:Lq-dim-ssm}) that
 \[
  D_q\nu_{p,k} = (1-1/k) D_q\mu_p.
 \]
 Since the assumption $r^{q-1}>\sum_{i=1}^m p_i^q$ says precisely that $D_q\mu_p>1$, we can fix a sufficiently large $k$ so that $D_q\nu_{p,k}>1$. We can therefore apply Theorem \ref{thm:projection-dim-irrational} to $\nu_{p,k}$ and conclude that $D_q(P_\beta\nu_{p,k})=1$ for \emph{all} $\beta$.

 On the other hand, it follows from Proposition \ref{prop:EK-projections} that there is a set $E_k\subset [0,\pi)$ with $\hdim E_k=0$, such that $P_\beta\eta_{p,k}\in\mathcal{D}_1$ for all $\beta\in [0,\pi)\setminus E_k$. Set $E=\bigcup_{k\in \N} E_k$.

 By linearity of convolving and projecting, $P_\beta\mu_p = P_\beta\nu_{p,k}*P_\beta\eta_{p,k}$. Theorem \ref{thm:abs-cont-convolution} now yields part (ii) of the theorem, and letting $q\searrow 1$ also part (i) in the irrational rotation case.
\end{proof}

\begin{proof}[Proof of Theorem \ref{mainthm:projection-sets}]
Let $A$ be a self-similar set on $\R^2$. We consider first the case $\hdim A>1$. Then $A$ contains a self-similar set $A'$ which satisfies the strong separation condition and still satisfies $\hdim A'>1$; this is folklore, see e.g. \cite[Lemma 3.4]{Orponen12}. In turn, $A'$ contains a self-similar subset $A''$ which is also homogeneous, see \cite[Proposition 6]{PeresShmerkin09} (although it is not stated explicitly, it is clear from the proof that the IFS generating $A''$ also satisfies the strong separation condition). The claim now follows by applying Theorem \ref{mainthm:projection-measures}(i) to the natural self-similar measure on $A''$.

For completeness we comment on the (essentially known) case $\hdim A\le 1$. When the IFS contains a scaled irrational rotation, it was proved in \cite[Theorem 5]{PeresShmerkin09} that there are no exceptions at all, i.e. $\hdim P_\beta A=\hdim A$ for all $\beta$. Otherwise, we can approximate from inside as above to reduce the statement to the case of a homogeneous, no-rotations, self-similar set with strong separation, which follows from \cite[Theorem 1.8]{Hochman13}.
\end{proof}

\subsection{Proof of Theorems \ref{mainthm:convolutions} and \ref{mainthm:convolutions-sets}}
\label{subsec:proof-convolutions}

The proof of our last two main theorems is not unlike those of Theorems \ref{mainthm:projection-measures} and \ref{mainthm:projection-sets}, except that there is a twist in the decomposition into ``large'' and ``Fourier decay'' parts. Here we need to appeal to Proposition \ref{prop:EK-convolutions} and the following result from \cite{NPS12}.

\begin{thm} \label{thm:convolution-ssm-dim}
Let $\mu_i=\mu(\lam_i,a_i,p_i)$ be two homogeneous self-similar measures on the real line satisfying the strong separation condition. Assume $\log\lam_2/\log\lam_1\notin\mathbb{Q}$. Then for any $q\in (1,2]$ and any $u\neq 0$,
\[
D_q(\mu_1* T_u\mu_2) = \min\left(D_q(\mu_1)+D_q(\mu_2) ,1\right).
\]
\end{thm}
A particular case of the above is \cite[Theorem 5.1]{NPS12} (which corresponds to the case where $\mu_1, \mu_2$ are natural measures on central Cantor sets). However, applying the remarks on the biased case and on general self-similar measures in \cite[Section 5]{NPS12}, the proof easily extends to this generality. To be more precise, the key cocycle estimate \cite[(3.1)]{NPS12} continues to hold in this generality (the proof of this requires only notational changes) and the conclusion of the proof from here is identical to that of \cite[Theorem 1.1]{NPS12}.

\begin{proof}[Proof of Theorem \ref{mainthm:projection-measures}]
Fix $\lam_i,m_i,a_i$ as in the statement. We consider the cases $\log |\lam_2|/\log|\lam_1|\in\mathbb{Q}$ and $\log |\lam_2|/\log|\lam_1|\notin\mathbb{Q}$ separately. We start with the rational case. In this case, there are integers $h_1, h_2$ such that $\lam_1^{h_1}=\lam_2^{h_2}>0$, so that, after iterating each IFS a suitable number of times (which does not affect the hypotheses of the theorem), we can assume $\lam_1=\lam_2>0$. This implies that, for any choice of $p_i\in\PP_{m_i}$,
\[
\mu(\lam,a_1,p_1) \times \mu(\lam,a_2,p_2) = \mu(\lam,\widetilde{a},\widetilde{p}),
\]
where $\widetilde{a}\in\R^{m_1 m_2}$ is defined (after relabeling) by $\widetilde{a}_{ij}=(a_i,a_j)$, and $\widetilde{p}\in\PP_{m_1 m_2}$ is given by $\widetilde{p}_{ij}=p_i p_j$. Since the family of convolutions $\mu(\lam,a_1,p_1) * T_u \mu(\lam,a_2,p_2)$ is, after smooth reparametrization, the family of orthogonal projections of $\mu(\lam,\widetilde{a},\widetilde{p})$ in non-principal directions, the claim (i) in the rational case follows from Theorem \ref{mainthm:parametrized} as in the proof of Theorem \ref{mainthm:projection-measures}.

Assume now that $\log|\lam_2|/\log|\lam_1|\notin\mathbb{Q}$. After iterating, we may assume that $\lam_i>0$. For simplicity, write $\mu^i_{p_i}=\mu(\lam_i,a_i,p_i)$. For a given $k$, we will decompose \emph{both} $\mu^1_{p_1}$ and $\mu^2_{p_2}$ as
\[
\mu^i_{p_i} = \nu^i_{p_i,k} * \eta^i_{p_i,k},
\]
where as usual $\nu^i_{p_i,k}$ is obtained from $\mu^i_{p_i}$ by skipping every $k$-th digit, and $\eta^i_{p_i,k}$ by keeping every $k$-th digit. By linearity of convolution,
\[
\mu^1_{p_1} * T_u \mu^2_{p_2} = \left(\nu^1_{p_1,k} * T_u \nu^2_{p_2,k} \right) * \left(\eta^1_{p_1,k}* T_u \eta^2_{p_2,k}\right)
\]
It follows from \eqref{eq:Lq-dim-ssm} and the assumption \eqref{eq:sum-of-corr-dims-g1} that if $k$ is large enough, then
\[
D_q(\nu^1_{p_1,k})+ D_q(\nu^2_{p_2,k}) > 1.
\]
Thus Theorem \ref{thm:convolution-ssm-dim} yields that $D_q(\nu^1_{p_1,k} * T_u \nu^2_{p_2,k})=1$ for all $u\neq 0$ (here we use the assumption $q\le 2$).

To handle $\eta^1_{p_1,k}* T_u \eta^2_{p_2,k}$, we appeal to Proposition \ref{prop:EK-convolutions}, which guarantees that these measures are in $\mathcal{D}_1$ for all $u$ outside of a set $E_k$ of zero Hausdorff dimension (independent of $p_1,p_2$). Set $E=\bigcup_k E_k$. The proof of part (ii) is now concluded in view of Theorem \ref{thm:abs-cont-convolution}, and part (i) in the irrational case follows by letting $q\searrow 1$.
\end{proof}

We conclude the paper with the short derivation of Theorem \ref{mainthm:convolutions-sets}.
\begin{proof}[Proof of Theorem \ref{mainthm:convolutions-sets}]
Let $A_1, A_2$ be as in the statement. As in the proof of Theorem \ref{mainthm:projection-sets}, it follows e.g. from \cite[Lemma 3.4]{Orponen12} and \cite[Proposition 6]{PeresShmerkin09} that there are self-similar sets $A'_i\subset A_i$ ($i=1,2$), which are homogeneous and satisfy the strong separation condition, and of dimensions arbitrarily close to those of $A_i$; in particular, we may assume $\hdim A'_1+\hdim A'_2>1$. The theorem now follows at once from Theorem \ref{mainthm:convolutions} applied to the natural self-similar measures on $A'_1,A'_2$.
\end{proof}


\end{document}